\documentclass[10pt]{article}
\usepackage{geometry,amsmath,amssymb,amsthm,graphicx,epsfig}
\usepackage{subcaption,enumerate,dcolumn,latexsym,epstopdf,color,url}
\usepackage[graph]{xy}
\usepackage[colorlinks=true,linkcolor=blue,citecolor=blue]{hyperref}
\usepackage[usenames,dvipsnames,svgnames,table]{xcolor}
\usepackage{ulem}
\usepackage{hyperref}
\numberwithin{equation}{section}
\usepackage{multirow}
\usepackage{tabularx}

\usepackage{pdfsync,comment,a4wide}
\usepackage{enumitem}
\usepackage[titletoc,title]{appendix}
\newtheorem{theorem}{Theorem} 
\newtheorem{proposition}{Proposition} 
\newtheorem{lemma}{Lemma} 
\theoremstyle{definition}
\newtheorem{remark}{Remark} 
\newtheorem{definition}{Definition} 
\newtheorem{assumption}{Assumption}

\renewcommand{\em}[1]{\normalem \em{#1}}
\renewcommand{\emph}[1]{\normalem \emph{#1}}
\newcommand{\ie}{{\em i.e.}, }
\newcommand{\eg}{{\em e.g.}, }
\newcommand{\cf}{{\em cf.\ }}

\newcommand{\nn}{\mathbb{N}} 
\newcommand{\real}{\mathbb{R}} 
\newcommand{\norm}[1]{\left\Vert {#1} \right\Vert} 
\newcommand{\erl}{\left(-\infty , +\infty\right]} 
\newcommand{\dom}[1]{\mathrm{dom}\,{#1}} 
\newcommand{\idom}[1]{\mathrm{int\,dom}\,{#1}} 

\DeclareMathOperator*{\argmin}{\arg\!\min}

\newcommand{\dist}{\mathrm{dist}} 
\newcommand{\act}[1]{\left\langle {#1} \right\rangle} 
\newcommand{\seq}[2]{\left\{{#1}_{{#2}}\right\}_{{#2} \in \mathbb{N}}}
\newcommand{\Seq}[2]{\left\{{#1}^{{#2}}\right\}_{{#2} \in \mathbb{N}}}

\newcommand{\ba}{{\bf a}}

\newcommand{\bo}{{\bf 0}}

\newcommand{\bu}{{\bf u}}
\newcommand{\bv}{{\bf v}}
\newcommand{\bx}{{\bf x}}

\newcommand{\by}{{\bf y}}
\newcommand{\bz}{{\bf z}}
\newcommand{\bb}{{\bf b}}

\newcommand{\bw}{{\bf w}}

\usepackage{mathtools}

\def\red{\color{red}}

\title{Dynamic FISTA for Convex Composite Bi-Level Optimization}
\author{Roey Merchav\thanks{Faculty of Industrial Engineering and Management, Technion--Israel Institute of Technology, Haifa 3200003, Israel. Email: merhav.roey@campus.technion.ac.il.} \and Shoham Sabach\stepcounter{footnote}\thanks{Faculty of Industrial Engineering and Management, Technion--Israel Institute of Technology, Haifa 3200003, Israel. Email: ssabach@technion.ac.il. This research was partially supported by the Israel Science Foundation (ISF) grant 2480-21.} \and Marc Teboulle\thanks{School of Mathematical Sciences, Tel-Aviv University, Ramat-Aviv 69978, Israel. E-mail: teboulle@post.tau.ac.il. This research was partially supported by the Israel Science Foundation (ISF) grant 2619-20.}}
\date{ }

\begin{document}

\maketitle

\begin{abstract}
	In this paper, we study convex bi-level optimization problems where both the inner and outer levels are given as a composite convex minimization. We propose the Fast Bi-level Proximal Gradient (FBi-PG) algorithm, which can be interpreted as applying FISTA to a dynamic regularized composite objective function. The dynamic nature of the regularization parameters allows to achieve an optimal fast convergence rate of $O(1/k^{2})$ in terms of the inner objective function. This is the fastest known convergence rate under no additional restrictive assumptions. We also show that FBi-PG achieves sub-linear simultaneous rates in terms of both the inner and outer objective functions. Moreover, we show that under an H\"olderian type error bound assumption on the inner objective function, the FBi-PG algorithm achieves improved simultaneous rates and converges to an optimal solution of the bi-level optimization problem. Finally, we present numerical experiments demonstrating the performance of the proposed scheme.
\end{abstract}

\smallskip


\section{Introduction}
	Bi-level optimization emerges as a dynamic and evolving research field within the broad optimization area. In this paper, we consider the bi-level convex minimization problem which aims at minimizing a convex function, called an {\em outer objective function} over the optimal solution set of another convex minimization of a function called the {\em inner objective}. We are interested in studying a convex bi-level model where {\em both} objectives in the inner and outer levels are instances of composite convex minimization problems. Our focus is on achieving faster convergence rates and corresponding pointwise global convergence for a first-order algorithm to this important class of bi-level composite convex optimization problems.
\medskip

	Following the terminology of inner and outer problems, we are interested in considering an outer problem given by the classical convex composite model
	\begin{equation} \label{Prob:OP} \tag{OP}
		\min_{\bx \in X^{\ast}} \left\{ \omega\left(\bx\right):= \sigma\left(\bx\right) + \psi\left(\bx\right) \right\},
	\end{equation}
	where $X^{\ast}$ is the, assumed non-empty, set of minimizers of the inner level problem which is also described via another convex composite model
	\begin{equation} \label{Prob:P} \tag{P}
		\min_{\bx \in \real^{n}} \left\{ \varphi\left(\bx\right) := f\left(\bx\right) + g\left(\bx\right) \right\}.
	\end{equation}
 	The functions $f , \sigma : \real^{n} \rightarrow \real$ are smooth and convex, while the functions $g , \psi : \real^{n} \rightarrow \erl$ are non-smooth and convex. We will provide the exact details of the problem's data in the forthcoming section. Moreover, we call $\varphi$ and $\omega$ the inner and outer objective functions, respectively. Throughout the paper, we denote by $X' \subset X^{\ast}$ the set of all optimal solutions of the outer optimization problem \eqref{Prob:OP}.
\medskip

	This type of bi-level minimization problems, is also known sometimes under the name {\em simple bi-level programming} as being a particular case of the more general bi-level programming model which involves, in both the inner/outer problems, parametric representation featuring two variables. This simpler model is however very flexible and as such relevant in modelling various modern applications; see, for instance, \cite{DZ2020} for a thorough overview of bi-level optimization problems as well as their applications, and the literature therein.
\medskip

	The implicit constraints imposed by the inner minimization problem make such problems difficult for developing simple and efficient algorithms. In recent years, there has been a growing interest in deriving theoretical guarantees beyond mere convergence results for convex bi-level optimization problems under various type of assumptions on the problem's data. For some of the most recent advances in this area, we refer the readers to \cite{SS2017} and the very recent paper \cite{MS2023}, which both also provide ample details on the past and current developments. Despite many of these interesting results, a grand challenge in this field remains open: {\it developing a first-order algorithm capable of achieving fast convergence rates for convex bi-level optimization problems} akin to the existing faster rates of first-order schemes for the traditional (single level) convex composite minimization \cite{N1983,AT2006,BT2009}.
\medskip
	
	Motivated by the remarkable success of fast first-order algorithms for traditional convex composite optimization problems just alluded above, our objective is to derive similar accelerated rates for the more challenging setting of bi-level convex composite optimization problems. Toward this goal, we propose employing the Tikhonov regularization technique, which involves approaching bi-level optimization as the minimization of either a single {\em fixed} regularized (penalized) or a {\em dynamic} regularized composite convex objective. The convex composite structure of the regularized objective thus naturally call for applying any fast first-order scheme on it. However, at this point we faced a hurdle. As we shall see below, we cannot apply directly existing convergence rate results due to either the lack of knowledge of the fixed regularization parameter or the dynamic change of the regularized composite objective. Thus, to tackle this inherent difficulty, a fresh analysis/approach is required.
\medskip

\noindent {\textbf{Literature.}}
	Before concluding this section, we briefly discuss several recent works related to our objectives. The literature on algorithms for simple bi-level optimization problems can broadly be categorized into two classes: {\em direct} methods and {\em nested} methods. The key distinction lies in computational complexity. Roughly speaking, direct methods avoid solving an inner optimization problem at each iteration, typically by considering an adequate regularized problem, whereas nested methods use approximation requiring  additional computations involving another optimization problem.
\medskip

	Given our focus on developing the simplest possible algorithms, we do not discuss nested approaches in detail. For instance, recent works such as \cite{JAMH2023, CJYM2024, WSJ2024, ZCXZ2024} propose nested algorithms with convergence guarantees for bi-level problems where the inner constraint set $Z$ is compact and convex. However, compactness is essential to their analysis, and the guarantees generally fail without it. Moreover, these methods require computing projections onto the intersection of $Z$ and a hyperplane, a task that is often computationally infeasible in practice.
\medskip

	Turning to direct methods, we highlight the recent works \cite{CSJW2024, SBY2024}, which address a regularized version of the problem using a fixed regularization parameter determined by a pre-specified iteration count $K$. While this approach enables a trivial analysis (see more details below), it does not guarantee recovery of a solution to the original bi-level problem.\footnote{See our numerical section for a concrete example illustrating this limitation in a real-world application.} More generally, as recognized in the literature, choosing an appropriate regularization parameter is a fundamentally difficult and often impractical task. Although convergence analysis under fixed regularization is typically straightforward (as we discuss below), it offers limited insight into the solution of the original bi-level formulation.
\medskip

	The paper is organized as follows. The subsequent section will comprehensively discuss the Tikhonov technique and its strong connections to bi-level optimization problems. This will lead us to propose the Fast Bi-level Proximal Gradient (FBi-PG) algorithm that can be simply interpreted as applying FISTA \cite{BT2009} to the dynamic regularized objective. The algorithm FBi-PG is equipped with a {\em dynamic} update of the regularization parameter, which is itself characterized by the choice of another fixed parameter allowing us to devise accelerated schemes. This is developed in Section \ref{Sec:Analysis}, where we will prove our main rate of convergence results. These analyses highlight the significance and crucial roles played by the dynamic regularization parameter, which is user-determined in regulating the performance of the FBi-PG scheme, a method which (for the first time) is proven to be capable of achieving the fast convergence rate of $O(1/k^{2})$ in terms of the inner optimization problem, see Theorem \ref{T:FastRateIn}. Additionally, we establish convergence rates concerning the outer objective function, and simultaneous rates concerning both functions. In the classical context of convex composite minimization, it is well-known that the sequence generated by FISTA does not necessarily converge to an optimal solution. This issue was however successfully addressed in \cite{CD2015} through a smart analysis of a variant of FISTA. The natural question is then if such a result can be extended to the more involved convex composite bi-level setting. In Section \ref{Sec:ConvAnal}, we positively address this question by adapting and extending the analysis of \cite{CD2015} to this case, and under an H\"{o}lder type error bound assumption \cite{BNPS2015}, we prove the convergence of the generated sequences by FBi-PG to an optimal solution of the bi-level problem \eqref{Prob:OP}. Finally, in Section \ref{Sec:numeric} we present numerical experiments  demonstrating the significant rate/improvement achieved by FBi-PG in comparison to other current state of the art schemes.
 	
\section{The FBi Proximal Gradient Algorithm} \label{Sec:Algo}
	
\subsection{Motivation and Difficulties} \label{SSec:Motivation}
	Several algorithmic approaches exist for addressing bi-level optimization problems. One of the earliest ideas dates back to the book of Tikhonov and Arsenin \cite{TA77-B}, who introduced what is now known as the Tikhonov regularization technique. In their original work, the goal was to find a minimal norm solution (meaning the outer objective function is $\omega \equiv \sigma = \norm{\cdot}^{2}$) to a linear optimization problem. Its primary advantage lies in simplifying the bi-level problem into a single regularized formulation, which avoids handling two optimization problems. In the context of our paper, it means to tackle the following penalized problem
	\begin{equation} \label{Prob:Tik} \tag{Pen$_\alpha$}
		\min_{\bx \in \real^{n}} \left\{ F_{\alpha}\left(\bx\right) \equiv \varphi\left(\bx\right) + \alpha\omega\left(\bx\right) \right\},
	\end{equation}
	where $\alpha > 0$ is the regularization parameter. In this context, we mention the work \cite{FT2008}, which studies among other questions, the connections between bi-level problems and their regularized counterpart in a general convex setting. They proved that there is a small enough $\alpha > 0$ such that an optimal solution of the regularized problem \eqref{Prob:Tik} is an optimal solution of the bi-level problem \eqref{Prob:OP}. Thus, a seemingly natural approach to obtain rates of convergence for the convex composite bi-level problem, based on the later idea of a-priori fixing the regularization parameter $\alpha$, would be to simply apply any first-order scheme on the regularized objective function $F_{\alpha}$, for some predefined parameter $\alpha$, to obtain rates on the {\em regularized} objective function $F_{\alpha}$. This approach has recently been formalized in \cite{CSJW2024, SBY2024}, though the resulting analysis is straightforward and follows directly from standard results on first-order methods. Indeed, for instance, an application of FISTA \cite[Theorem 4.4]{BT2009} on minimizing $F_{\alpha}$ with say an $L_{\alpha}$-Lipschitz gradient part of the objective will generate a sequence $\Seq{\bx}{k}$ which satisfies
	\begin{equation} \label{Fista}
		F_{\alpha}\left(\bx^{k}\right) - F_{\alpha}\left(\bx'\right) \leq \frac{2L_{\alpha}\norm{\bx^{0} - \bx'}^{2}}{\left(k + 1\right)^{2}}, \qquad \forall \, \bx' \in X'.
	\end{equation}
 	Using the definition of $F_{\alpha}$, and since $\varphi\left(\bx^{k}\right) - \varphi\left(\bx'\right) \geq 0$ for all $k \in \nn$, one immediately deduce the following fast rate of convergence for the outer problem
 	\begin{equation} \label{Fista-Omg}
 		\omega\left(\bx^{k}\right) - \omega\left(\bx'\right) \leq \frac{2L_{\alpha}\norm{\bx^{0} - \bx'}^{2}}{\alpha\left(k + 1\right)^{2}}, \qquad \forall \, \bx' \in X'.
 	\end{equation}
 	At this juncture, it is important to note that a generic result can be easily proved for the case of fixed regularizing parameter by showing that independently of the structure of the bi-level optimization problem and for any single-level optimization algorithm the arguments above remain valid. For sake of completeness, the proof of this fact appears in Appendix A. 
 Unfortunately, finding/knowing the value of a fixed regularization parameter $\alpha$, for which a solution of the regularized problem recovers an optimal solution of the bi-level problem, remains itself a challenging task. Furthermore, finding a corresponding rate for the inner optimization problem also depends on the unknown value of $\alpha$.  Therefore, obtaining rates of convergence for bi-level optimization problems {\em cannot} rely on existing results for single level optimization problems, and requires a fresh specially designed approach.
\medskip
 	
 	The first work to provide an efficient use of the Tikhonov regularization technique tailored made for bi-level optimization problems, is \cite{S07}. It was confined to the convex bi-level problem with a smooth and convex outer objective $\omega$ and with the inner objective function $\varphi$ being the sum of a smooth and convex function and an indicator of a closed and convex set. The resulting projected gradient algorithm applied on $F_{\alpha}$ proposed in \cite{S07} suggests that at each iteration $k$ we regularize the outer objective function $\omega$ with a {\em dynamic} parameter $\alpha_{k}$, and it is proved that if $\alpha_{k} \rightarrow 0$ as $k \rightarrow \infty$ as long as $\sum_{k = 1}^{\infty} \alpha_{k} = \infty$, then any limit point of the generated sequence is an optimal solution of the bi-level optimization problem. Therefore, following \cite{S07}, using a dynamic regularization parameter at each iteration seems to be the right approach to overcome the lack of knowledge of the ``right" regularization parameter $\alpha$. Yet, one of the interesting open questions about this approach and which we address in this work pertains to establish convergence rates, and in particular fast convergent rates, as well as the corresponding pointwise global converge results for the more general convex composite bi-level problem \eqref{Prob:OP}.
\medskip

	A first attempt in this direction was taken in the thesis \cite{T2021-T}. In that thesis, the fast scheme  based on the accelerated scheme of \cite{AT2006} (see also \cite{T2018} and references therein), which is different from FISTA was used. However, that study did not cover the general convex composite bi-level problem \eqref{Prob:OP}, the derived fast rate of convergence result was limited only to the inner problem, and no simultaneous rate was established. Moreover, pointwise global convergence was not derived. Inspired by \cite{T2021-T}, and using FISTA as our basis, here we are able to go much beyond \cite{T2021-T} by significantly extending both the rate and the global convergence results under various scenarios; see below for a description of our main results.	

\subsection{Problem's Statement and the FBi-PG Algorithm} \label{SSec:Algo}
	Motivated by the above discussion, in this paper, we extend the framework of FISTA proposed in \cite{BT2009} to introduce the Fast Bi-level Proximal Gradient (FBi-PG) algorithm designed specifically for solving convex bi-level optimization problems. Our approach incorporates the Tikhonov regularization technique, utilizing a {\em dynamically} updated regularization parameter to enhance the algorithm's performance in achieving a fast scheme. More precisely, given $\gamma > 0$, we will use the following sequence of regularizing parameters $\seq{\alpha}{k}$, which is defined for an $a \in \nn$ such that $a \geq 2$, by
	\begin{equation} \label{Alphak}
		\alpha_{k} = \left(k + a\right)^{-\gamma}, \quad k \in \nn.
	\end{equation}		
	For the simplicity of the developments below we define $\alpha_{-1} := 0$. Note that we could have set $a \equiv 1$. The additional parameter $a$ is already introduced here in order to simplify/unify our presentation, which also allows us to handle the pointwise global convergence, see Section \ref{Sec:ConvAnal}.
\medskip
	
	Before presenting the algorithm, considering the inner and outer optimization problems \eqref{Prob:P} and \eqref{Prob:OP}, throughout the paper we will use the following standard assumption on the problem's data.
	\begin{assumption} \label{A:AssumptionA}
		\begin{itemize}
    		\item[(i)] The functions $f , \sigma : \real^{n} \rightarrow \real$ are convex, continuously differentiable having Lipschitz continuous gradients with the Lipschitz parameters $\beta_{f}$ and $\beta_{\sigma}$, respectively. We denote $\beta \equiv \beta_{f} + \beta_{\sigma} \in \real_{++}$.
    		\item[(ii)] The functions $g , \psi : \real^{n} \rightarrow (-\infty , \infty]$ are proper, lower semicontinuous and convex.
    		\item[(iii)] The set $X^{\ast}$ of all optimal solutions of problem \eqref{Prob:P} is non-empty.
    		\item[(iv)] $\inf_{\bx \in \real^{n}} \omega\left(\bx\right) \equiv \omega^{\ast} > -\infty$.
		\end{itemize}
	\end{assumption}	
	Moreover, throughout the paper we systematically define, for all $k \in \nn$, the smooth regularized function $f_{k} \equiv f + \alpha_{k}\sigma$, the non-smooth regularized function $g_{k} \equiv g + \alpha_{k}\psi$, and the regularized function $F_{k} \equiv \varphi + \alpha_{k}\omega$.
	\begin{remark} \label{R:Lipfk}
		The functions $f$ and $\sigma$ having Lipschitz continuous gradients with the Lipschitz parameters $\beta_{f}$ and $\beta_{\sigma}$, respectively. Hence, for any $k \in \nn$, the function $f_{k}$ also has a Lipschitz continuous gradient with the parameter $\beta_{f} + \alpha_{k}\beta_{\sigma}$. Since $\alpha_{k} \leq 1$ then $\beta_{f} + \alpha_{k}\beta_{\sigma} \leq \beta_{f} + \beta_{\sigma} = \beta$.
	\end{remark}
	Based on this remark we see that the regularized objective function $F_{k}$ also has a composite structure of the sum of a smooth function $f_{k}$ with a Lipschitz continuous gradient and a non-smooth function $g_{k}$. We can thus apply FISTA \cite{BT2009} on the dynamic objective $F_{k}$, $k \in \nn$, to produce the FBi-PG algorithm, which is described below.

	{\center\fbox{\parbox{16cm}{{\bf Fast Bi-level Proximal Gradient}
			\begin{itemize}
				\item[ ] {\bf Input:} $\gamma > 0$ and $2 \leq a \in \nn$.
				\item[ ] {\bf Initialization}: $\alpha_{-1} = t_{-1} = 0$, $t_{0} = 1$ and $\bx^{0} = \bx^{-1} \in \real^{n}$.
				\item[ ] {\bf General Step} ($k = 0, 1 , 2 , \ldots$):
					\begin{align}
						\alpha_{k} & = \left(k + a\right)^{-\gamma} \quad f_{k} \equiv f + \alpha_{k}\sigma \quad g_{k} \equiv g + \alpha_{k}\psi \label{Bi-AGUpdate:0} \\
						\by^{k} & = \bx^{k} + t^{-1}_{k}\left(t_{k - 1} - 1\right)\left(\bx^{k} - \bx^{k - 1}\right) \label{Bi-AGUpdate:1} \\
						\bx^{k + 1} & = \argmin_{\bx \in \real^{n}} \left\{ g_{k}\left(\bx\right) + \bx^{T}\nabla f_{k}\left(\by^{k}\right) + \frac{\beta}{2}\norm{\bx - \by^{k}}^{2} \right\} \label{Bi-AGUpdate:2} \\
						t_{k + 1} & = \left(1 + \sqrt{1 + 4t_{k}^{2}}\right)/2 \label{Bi-AGUpdate:3}
				\end{align}
			\end{itemize}\vspace{-0.1in}}}}
\vspace{0.15in}

	The FBi-PG algorithm is associated with a parameter $\gamma > 0$, via the regularization parameter $\alpha_{k} = \left(k + a\right)^{-\gamma}$, which allows us to capture a broad range of behaviors within a unified framework. Throughout the paper, we establish three types of theoretical guarantees for the FBi-PG algorithm:
	\begin{itemize}
    	\item[(i)] individual convergence rates for the inner or outer problems;
    	\item[(ii)] simultaneous convergence rates for both levels;
    	\item[(iii)] pointwise convergence of the iterates to an optimal solution.
	\end{itemize}	
	Although simultaneous convergence is the ultimate goal in bi-level optimization, individual convergence guarantees remain valuable. They offer insight into the algorithm's behavior across regimes, reflect the historical progression of the field, which focused initially on inner rates, and often serve as foundational steps toward establishing full simultaneous results.
\medskip
	
	All our rate results are provided in Section \ref{Sec:Analysis}. We first present our results under the basic Assumption~\ref{A:AssumptionA}, and then show how they can be strengthened under a H\"olderian error bound for the inner objective function. To help the reader navigate the results, we provide the following summary table outlining the convergence guarantees, underlying assumptions, and corresponding theorems.

	\begin{table}[h]
		\centering
		\begin{tabular}{cccccc}
		\hline
		Parameter & Inner & Outer & Simul. & Relevant & Note  \\
		Regime & Convergence & Convergence & & Thms/Props & \\
\hline
      	$\gamma \in (0,1)$ & $O(1/k^\gamma)$  & $O(1/k^{2 - \gamma})$ & No  & Thm 2, Prop 5 & Outer best iterate \\
       	$\gamma = 1$       & $O(\log k / k)$  & $O(\log k / k)$       & Yes & Thm 3         & Both best iterate \\
        $\gamma \in (1,2)$ & $O(1/k^\gamma)$  & $O(1/k^{2 - \gamma})$ & No  & Thm 2, Prop 5 & Outer best iterate \\
    	$\gamma \in (1,2)$ & $ O(1/k^\gamma)$ & $O(1/k^{2 - \gamma})$ & Yes & Thm 4 and 5   & H\"olderian EB \\
        $\gamma \geq 2$    & $O(1/k^{2})$     & N/A                   & No  & Thm 1         & \\
\hline
		\end{tabular}
		\caption{Summary of convergence rate results under various parameter regimes.}
	\end{table}
\medskip

	Before concluding this section, in the following remark we briefly discuss the algorithm’s main computational step.
	\begin{remark}[Computation of the prox of the sum of two functions -- Lifted reformulation] \label{R:Lift}
		The FBi-PG algorithm needs to compute, at each iteration $k \in \nn$, the proximal mapping of the sum  $g + \alpha_{k}\psi$, where the two functions $g, \psi$ are assumed prox friendly. Most often this is a difficult and computationally demanding task, even when both functions are prox friendly. The nice observation made in \cite[Remark 2.1, page 50]{LTVP2025} provides a way that circumvents this difficulty. The lifting technique proposed in \cite{LTVP2025}, which allows us to eliminate this computational bottleneck by reformulating the bi-level optimization problem as follows
		\begin{equation*}
			\min_{\bw \in W^{\ast}} \left\{ {\bar \omega}\left(\bw\right) \equiv \sigma\left(\bx\right) + \psi\left(\bz\right) \, : \, \bw \in \argmin_{\bw \in \real^{n \times n}} \left\{ {\bar \varphi}\left(\bw\right) \equiv f\left(\bx\right) + g\left(\bx\right) + \norm{\bx - \bz}^{2}/2 \right\} \right\}.
		\end{equation*}
 		The equivalence immediately follows, since in the inner problem any minimizer $\bw' = \left(\bx' , \bz'\right)$ must satisfy $\bz’ = \bx’$, with $\bx'$ minimizer of $\varphi$. Thanks to this lifted reformulation, applying the Fbi-PG reduces now to compute the prox of the {\em separable} sum $g\left(\bx\right) + \alpha_{k}\psi\left(\bz\right)$, \ie it reduces to compute separately each of the prox friendly functions $g$ and $\psi$.
	\end{remark}


	
\section{Rate of Convergence Analysis of FBi-PG} \label{Sec:Analysis}
	The section is divided into three parts, where we start with a few key results, which will serve us in deriving the main results. In Section \ref{SSec:FastRate}, we will first prove (last iterate) individual inner rates for three different regimes: $0 < \gamma < 2$, $\gamma = 2$, and $\gamma > 2$. In the last regime, we obtain the fast inner rate, which is the best current known acceleration result. Then, we will provide (best iterate) outer rate of convergence result for $0 < \gamma < 2$. We conclude this part with our main result, which is a simultaneous (best iterate) rate of both the inner and outer problems for $\gamma = 1$. In Section \ref{SSec:ImprovedRate}, we will prove an improved (last iterate) simultaneous rates for $1 < \gamma < 2$ under the H\"olderian error bound condition for the inner objective function.


\subsection{Main Pillars of the Analysis} \label{SSec:Pillars}
	In order to obtain our basic inequality for the development of all the announced results, we will need the following classical proximal inequality (see \cite[Lemma 2.6]{BT10}).
	\begin{proposition} \label{P:TechProx}
		Let $s : \real^{n} \rightarrow \real$ be a convex and continuously differentiable function having a Lipschitz continuous gradient with the Lipschitz parameter $L > 0$, and let $q : \real^{n} \rightarrow (-\infty , \infty]$ be a proper, lower semicontinuous and convex function. For a given vector $\by \in \real^{n}$, we define
		\begin{equation*}
			\bx^{+} := \argmin_{\bx \in \real^{n}} \left\{ q\left(\bx\right) + \bx^{T}\nabla s\left(\by\right) + \frac{L}{2}\norm{\bx - \by}^{2} \right\}.
		\end{equation*}
		Then, for all $\bu \in \dom{q}$, we have
		\begin{equation*}
			q\left(\bx^{+}\right) + s\left(\bx^{+}\right) - \left(q\left(\bu\right) + s\left(\bu\right)\right) \leq \frac{L}{2}\left(\norm{\bu - \by}^{2} - \norm{\bu - \bx^{+}}^{2}\right).
		\end{equation*}		
	\end{proposition}
  	At each iteration $k \in \nn$, the steps of FBi-PG can be seen as a direct application of the well-known FISTA algorithm \cite{BT2009} on the regularized function $F_{k}$. The next technical result can be obtained from \eg \cite[Lemma 2.7]{BT10}. However, for the sake of completeness, we include its simple proof.
	\begin{proposition} \label{P:BasicIne}
		Let $\Seq{\bx}{k}$ be a sequence generated by FBi-PG. For all $k \in \nn$ and $\bx \in \real^{n}$, we have
		\begin{equation} \label{P:BasicIne:0}
			t_{k}^{2}\left(F_{k}\left(\bx^{k + 1}\right) - F_{k}\left(\bx\right)\right) + \frac{\beta}{2}\norm{\bz^{k + 1} - \bx}^{2} \leq \left(t_{k}^{2} - t_{k}\right)\left(F_{k}\left(\bx^{k}\right) - F_{k}\left(\bx\right)\right) + \frac{\beta}{2}\norm{\bz^{k} - \bx}^{2},
		\end{equation}
		where we define $\bz^{k} := \left(1 - t_{k - 1}\right)\bx^{k - 1} + t_{k - 1}\bx^{k}$.
	\end{proposition}
	\begin{proof}
		Let $k \in \nn$. From the $\bx^{k + 1}$ step (see \eqref{Bi-AGUpdate:2}), using Proposition \ref{P:TechProx} with $s\left(\cdot\right) = f_{k}\left(\cdot\right)$, $q\left(\cdot\right) = g_{k}\left(\cdot\right)$, $\by = \by^{k}$, $\bx^{+} = \bx^{k + 1}$, $\bu = \left(1 - t_{k}^{-1}\right)\bx^{k} + t_{k}^{-1}\bx$ and $L = \beta$ we get for any $\bx \in \real^{n}$ that (recall $F_{k} \equiv f_{k} + g_{k}$)
        \begin{align*}
        	F_{k}\left(\bx^{k + 1}\right) & \leq F_{k}\left(\left(1 - t_{k}^{-1}\right)\bx^{k} + t_{k}^{-1}\bx\right) +\frac{\beta}{2}\left(\norm{\left(1 - t_{k}^{-1}\right)\bx^{k} + t_{k}^{-1}\bx - \by^{k}}^{2} - \norm{\left(1 - t_{k}^{-1}\right)\bx^{k} + t_{k}^{-1}\bx - \bx^{k + 1}}^{2}\right) \\
         	& = F_{k}\left(\left(1 - t_{k}^{-1}\right)\bx^{k} + t_{k}^{-1}\bx\right) + \frac{\beta}{2t_{k}^{2}}\left(\norm{\bz^{k} - \bx}^{2} - \norm{\bz^{k + 1} - \bx}^{2}\right),
        \end{align*}
        where the last equality follows from the definitions of $\bz^{k}$ and $\by^{k}$ (see \eqref{Bi-AGUpdate:1}). Since $F_{k}$ is convex, the inequality above yields
        \begin{equation} \label{P:BasicIne:1}
            F_{k}\left(\bx^{k+1}\right) \leq \left(1 - t_{k}^{-1}\right)F_{k}\left(\bx^{k}\right) + t_{k}^{-1}F_{k}\left(\bx\right) + \frac{\beta}{2t_{k}^{2}}\left(\norm{\bz^{k} - \bx}^{2} - \norm{\bz^{k + 1} - \bx}^{2}\right).
        \end{equation}
        Subtracting $F_{k}\left(\bx\right)$ from both sides and multiplying both sides with $t_{k}^{2}$, we obtain after rearranging that
         \begin{equation*}
        	t_{k}^{2}\left(F_{k}\left(\bx^{k + 1}\right) - F_{k}\left(\bx\right)\right) + \frac{\beta}{2}\norm{\bz^{k + 1} - \bx}^{2} \leq \left(1 - t_{k}^{-1}\right)t_{k}^{2}\left(F_{k}\left(\bx^{k}\right) - F_{k}\left(\bx\right)\right) + \frac{\beta}{2}\norm{\bz^{k} - \bx}^{2},
        \end{equation*}
        which proves the desired result.
	\end{proof}
	From now on, we cannot any more apply the rate of convergence results of FISTA (\cite[Theorem 4.4]{BT2009} or \cite[Theorem 2.4]{BT10}) to obtain relevant faster rates for the algorithm FBi-PG, and we need a new line of analysis. Indeed, the {\em dynamic} nature of the regularized function $F_{k}$, $k \in \nn$, precludes the possibility of using directly the rates established for FISTA on $F_{k}$, unless one either pick a fixed parameter $\alpha_{k} \equiv \alpha > 0$, for all $k \in \nn$, or one fixes in advance the iteration counter at some predefined value $k \equiv K$. However, both strategies are not practically and theoretically viable. Indeed, as discussed earlier in Section \ref{Sec:Algo} and as also done in \cite{CSJW2024, SBY2024}, in the former case (among other troubles!) the value of the ``right" regularizing parameter $\alpha$ is in general not available, and in the latter case it is immediate to see that from \eqref{Fista-Omg} with $\alpha_{K} := \left(K + 1\right)^{-\gamma}$ (for some $0< \gamma < 2$) one obtains
	\begin{equation*}
 		\omega\left(\bx^{K}\right) - \omega\left(\bx'\right) \leq \frac{2\beta\norm{\bx^{0} - \bx'}^{2}}{\left(K + 1\right)^{2 - \gamma}}, \qquad \forall \, \bx' \in X',
	\end{equation*}
 	which shows that the fast rate is lost! Furthermore, in general, it makes no sense to fix in advance the number of iterations to be performed by the algorithm FBi-PG since, eventually in order to warrant convergence of the sequence generated by FBi-PG to an optimal solution of the bi-level optimization problem \eqref{Prob:OP}, we need to have the sequence of regularized parameters $\seq{\alpha}{k}$ to converge to $0$ as $k  \rightarrow \infty$. The numerical results presented in Section \ref{Sec:numeric} will further illustrate and corroborate this situation. The next results will be useful to overcome the alluded difficulties.
\medskip

	For convenience, we define for all $k \in \nn$
	\begin{equation} \label{deta}
		d_{k} := t_{k - 1}^{2} - \left(t_{k}^{2} - t_{k}\right)\; \text{and} \; \eta_{k} := t_{k - 1}^{2}\alpha_{k - 1} - \left(t_{k}^{2} - t_{k}\right)\alpha_{k}.
	\end{equation}
	We note that since $t_{k}^{2} - t_{k} \leq t_{k - 1}^{2}$, then $d_{k} = t_{k - 1}^{2} - \left(t_{k}^{2} - t_{k}\right) \geq 0$ for all $k \in \nn$. Likewise, using the definition of $\alpha_{k}$ (see \eqref{Bi-AGUpdate:0}) we have that $\alpha_{k} < \alpha_{k - 1}$ and thus $\eta_{k} = t_{k - 1}^{2}\alpha_{k - 1} - \left(t_{k}^{2} - t_{k}\right)\alpha_{k}\geq t_{k - 1}^{2}\left(\alpha_{k - 1} - \alpha_{k}\right) > 0$ for all $k \in \nn$.
    \begin{proposition} \label{P:SimIne}
    	Let $\Seq{\bx}{k}$ be a sequence generated by FBi-PG. For all $k \in \nn$ and $\bx' \in X'$, we have
    	\begin{equation} \label{C:SimIne:0}
    		t_{k - 1}^{2}\left(F_{k - 1}\left(\bx^{k}\right) - F_{k - 1}\left(\bx'\right)\right) \leq \frac{\beta}{2}\norm{\bx^{0} - \bx'}^{2} - \sum_{s = 0}^{k - 1} \eta_{s}\left(\omega\left(\bx^{s}\right) - \omega\left(\bx'\right)\right).
    	\end{equation}
    \end{proposition}
    \begin{proof}
    	From Proposition \ref{P:BasicIne} (see \eqref{P:BasicIne:0}) we have for all $k \in \nn$ and $\bx \in \real^{n}$ that
		\begin{equation} \label{P:SimIne:1}
			t_{k}^{2}\left(F_{k}\left(\bx^{k + 1}\right) - F_{k}\left(\bx\right)\right) + \frac{\beta}{2}\norm{\bz^{k + 1} - \bx}^{2} \leq \left(t_{k}^{2} - t_{k}\right)\left(F_{k}\left(\bx^{k}\right) - F_{k}\left(\bx\right)\right) + \frac{\beta}{2}\norm{\bz^{k} - \bx}^{2}.
		\end{equation}    	    	
    	 Now, using the definition of $F_{k} \equiv \varphi + \alpha_{k} \omega$, we get that
    	\begin{align}
    		\left(t_{k}^{2} - t_{k}\right)\left(F_{k}\left(\bx^{k}\right) - F_{k}\left(\bx\right)\right) & = \left(t_{k}^{2} - t_{k}\right)\left(\varphi\left(\bx^{k}\right) - \varphi\left(\bx\right)\right) + \left(t_{k}^{2} - t_{k}\right)\alpha_{k}\left(\omega\left(\bx^{k}\right) - \omega\left(\bx\right)\right) \nonumber \\
			& = t_{k - 1}^{2}\left(\varphi\left(\bx^{k}\right) - \varphi\left(\bx\right)\right) + (t_{k}^{2} - t_{k} - t_{k-1}^{2}) \left(\varphi\left(\bx^{k}\right) - \varphi\left(\bx\right)\right) \nonumber \\ 			
			& + \left(\alpha_{k - 1}t_{k - 1}^{2} - \eta_{k}\right)\left(\omega\left(\bx^{k}\right) - \omega\left(\bx\right)\right) \nonumber \\
			& = t_{k - 1}^{2}\left(F_{k - 1}\left(\bx^{k}\right) - F_{k - 1}\left(\bx\right)\right) - d_{k} \left(\varphi\left(\bx^{k}\right) - \varphi\left(\bx\right)\right)  \nonumber \\
            & - \eta_{k}\left(\omega\left(\bx^{k}\right) - \omega\left(\bx\right)\right), \label{P:SimIne:2}
		\end{align}    	
    	where the second equality follows from the fact that $\left(t_{k}^{2} - t_{k}\right)\alpha_{k} = t_{k - 1}^{2}\alpha_{k - 1} - \eta_{k}$ while the last equality follows from the definition of $F_{k - 1} \equiv \varphi + \alpha_{k - 1}\omega$ and the fact that $t_{k}^{2} - t_{k} - t_{k-1}^{2} = -d_{k}$. Combining it with \eqref{P:SimIne:1} for $k = s$ yields that
		\begin{align}
			t_{s}^{2}\left(F_{s}\left(\bx^{s + 1}\right) - F_{s}\left(\bx\right)\right) + \frac{\beta}{2}\norm{\bz^{s + 1} - \bx}^{2} & \leq t_{s - 1}^{2}\left(F_{s - 1}\left(\bx^{s}\right) - F_{s - 1}\left(\bx\right)\right) + \frac{\beta}{2}\norm{\bz^{s} - \bx}^{2} \nonumber \\
			& - d_{s} \left(\varphi\left(\bx^{s}\right) - \varphi\left(\bx\right)\right) - \eta_{s}\left(\omega\left(\bx^{s}\right) - \omega\left(\bx\right)\right). \label{P:SimIne:3}
		\end{align}    	    	
        Summing \eqref{P:SimIne:3} over $s = 0 , 1 , \ldots , k - 1$ results in
        \begin{align}
         	t_{k - 1}^{2}\left(F_{k - 1}\left(\bx^{k}\right) - F_{k - 1}\left(\bx\right)\right) + \frac{\beta}{2}\norm{\bz^{k} - \bx}^{2} & \leq t_{-1}^{2}\left(F_{-1}\left(\bx^{0}\right) - F_{-1}\left(\bx\right)\right) + \frac{\beta}{2}\norm{\bz^{0} - \bx}^{2} \nonumber \\
         	& - \sum_{s = 0}^{k - 1} d_{s}\left(\varphi\left(\bx^{s}\right) - \varphi\left(\bx\right)\right) - \sum_{s = 0}^{k - 1} \eta_{s}\left(\omega\left(\bx^{s}\right) - \omega\left(\bx\right)\right) \nonumber \\
         	& = \frac{\beta}{2}\norm{\bx^{0} - \bx}^{2} - \sum_{s = 0}^{k - 1} d_{s}\left(\varphi\left(\bx^{s}\right) - \varphi\left(\bx\right)\right) \nonumber \\
          & - \sum_{s = 0}^{k - 1} \eta_{s}\left(\omega\left(\bx^{s}\right) - \omega\left(\bx\right)\right), \label{P:SimIne:4}
		\end{align}
		where the equality follows from the fact that since $\varphi$ is proper (see Assumption \ref{A:AssumptionA}) and $\alpha_{-1} = 0$, it follows that $F_{-1} \equiv \varphi + \alpha_{-1}\omega = \varphi < \infty$, along with the facts that $t_{-1} = 0$ and $\bx^{0} = \bz^{0}$. Eliminating the non-negative term $\beta\norm{\bz^{k} - \bx}^{2}/2$ on the left-hand side, and plugging $\bx = \bx' \in X'$ in \eqref{P:SimIne:4} it follows that
		\begin{equation} \label{P:SimIne:5}
    		t_{k - 1}^{2}\left(F_{k - 1}\left(\bx^{k}\right) - F_{k - 1}\left(\bx'\right)\right) \leq \frac{\beta}{2}\norm{\bx^{0} - \bx'}^{2} - \sum_{s = 0}^{k - 1} d_{s}\left(\varphi\left(\bx^{s}\right) - \varphi\left(\bx'\right)\right) - \sum_{s = 0}^{k - 1} \eta_{s}\left(\omega\left(\bx^{s}\right) - \omega\left(\bx'\right)\right).
		\end{equation}
		Since $\bx' \in X'$ we have that $\varphi\left(\bx^{s}\right) - \varphi\left(\bx'\right) \geq 0$ for all $s \in \nn$, and since $d_{s} \geq 0$ for all $s \in \nn$, the desired result follows.
	\end{proof}
	We conclude this part with the following result.
	\begin{proposition} \label{P:MainIne}
		Let $\Seq{\bx}{k}$ be a sequence generated by FBi-PG. For all $k \in \nn$ and $\bx' \in X'$, we have
		\begin{equation} \label{P:MainIne:0}
        	t_{k - 1}^{2}\left(\varphi\left(\bx^{k}\right) - \varphi\left(\bx'\right)\right) \leq \frac{\beta}{2}\norm{\bx^{0} - \bx'}^{2} + \left(\omega\left(\bx'\right) - \omega^{\ast}\right)\sum_{s = 0}^{k - 1} \alpha_{s}t_{s}.
        \end{equation}
	\end{proposition}
	\begin{proof}
		Let $k \in \nn$. For the simplicity of the proof we define the following non-negative quantities (recall that $\omega^{\ast} = \inf_{\bx \in \real^{n}} \omega\left(\bx\right)$) $\Omega_{max} := \omega\left(\bx'\right) - \omega^{\ast}$ and $\Omega_{k} := \omega\left(\bx^{k}\right) - \omega^{\ast}$, $k \in \nn$. Using the definition of $F_{k - 1} \equiv \varphi + \alpha_{k - 1}\omega$ in Proposition \ref{P:SimIne} (see \eqref{C:SimIne:0}) we obtain
 		\begin{equation*}
   			t_{k -1}^{2}\left(\varphi\left(\bx^{k}\right) - \varphi\left(\bx'\right)\right) + \alpha_{k - 1}t_{k - 1}^{2}\left(\Omega_k-\Omega_{\max}\right) \leq \frac{\beta}{2}\norm{\bx^{0} - \bx'}^{2}- \sum_{s = 0}^{k - 1} \eta_{s}\left( \Omega_k-\Omega_{\max}\right).
    	\end{equation*}
    	Now, since $\eta_{k} > 0$ and $\Omega_{k} \geq 0$ for all $k \in \nn$, it follows from the above inequality that
    	\begin{equation} \label{P:MainIne:1}
     		t_{k -1}^{2}\left(\varphi\left(\bx^{k}\right) - \varphi\left(\bx'\right)\right) - \alpha_{k - 1}t_{k - 1}^{2}\Omega_{\max} \leq \frac{\beta}{2}\norm{\bx^{0} - \bx'}^{2}+ \Omega_{\max}\sum_{s = 0}^{k - 1} \eta_{s}.
    	\end{equation}
 	To complete the proof, note that from the definition of $\eta_{k} = \alpha_{k - 1}t_{k - 1}^{2} - \left(t_{k}^{2} - t_{k}\right)\alpha_{k}$, we have that
 		\begin{equation*}
    		\sum_{s = 0}^{k - 1} \eta_{s} = \sum_{s = 0}^{k - 1} \left(\alpha_{s - 1}t_{s - 1}^{2} - \alpha_{s}t_{s}^{2}\right)
    + \sum_{s = 0}^{k - 1} \alpha_{s}t_{s} = \alpha_{-1}t_{-1}^{2} - \alpha_{k - 1}t_{k - 1}^{2} + \sum_{s = 0}^{k - 1}
    \alpha_{s}t_{s} = -\alpha_{k - 1}t_{k - 1}^{2} + \sum_{s = 0}^{k - 1} \alpha_{s}t_{s},
 		\end{equation*}
    	where the last equality follows from the fact that $\alpha_{-1} = t_{-1} = 0$. Plugging it in \eqref{P:MainIne:1}, the desired result follows.
    \end{proof} 	
	In this section, we established several results for FBi-PG that will serve as the foundation for our main convergence analysis. As noted earlier, our algorithm can be interpreted as a dynamic version of FISTA, where the objective functions vary with each iteration. This dynamic nature introduces significant challenges in the analysis, which deviates notably from the classical FISTA framework. A closer look at the basic inequality in Proposition \ref{P:BasicIne} reveals that \eqref{P:BasicIne:0} nearly exhibits a telescoping structure. However, the difficulty arises from the iteration-dependent objective function $F_{k}$. In Proposition \ref{P:SimIne}, we address this issue, but at the cost of introducing an additional term in \eqref{C:SimIne:0},
 which subsequently implies a more involved analysis and influences the strength of the results we can obtain.
	
\subsection{Rates of Convergence of FBi-PG} \label{SSec:FastRate}
	In this part, we will prove three types of rate of convergence results: on the inner objective function, on the outer objective function, and simultaneous on both the inner and outer objective functions. All the results under the standard Assumption \ref{A:AssumptionA}.
\medskip
	
	In order to simplify our developments below we will take (recall that $a \in \nn$ such that $a \geq 2$)
	\begin{equation} \label{D:Tk}
		t_{k} = \frac{k + a}{a}, \quad k \in \nn.
	\end{equation}
	Moreover, this choice is not accidental and is also crucial for the pointwise convergence analysis developed in Section \ref{Sec:ConvAnal}. However, it should be noted that all our results of Section \ref{SSec:Pillars} are valid for any sequence $\seq{t}{k}$, which satisfies $t_{k}^{2} - t_{k} \leq t_{k - 1}^{2}$ for all $k \in \nn$. It is easy to check that our choice of $t_{k}$ in \eqref{D:Tk} satisfies this requirement. Indeed, $t_{k - 1}^{2} - t_{k}^{2} + t_{k} = \left(\left(a - 1\right)^{2} + \left(a - 2\right)k\right)/a^{2} > 0$, with $a\geq 2$.
\medskip	
	
	We begin with two technical lemmas to simplify the proofs afterwards. The first one was proved in \cite[Lemma 4.3]{MS2023}.
	\begin{lemma} \label{L:TechSum}
		Let $n_{1}\in \nn$ and $n_{2} \in \nn \cup \{\infty\}$ be such that $n_{1} \leq n_{2}$. Then, for any $0<r<1$, we have
		\begin{equation*}
    		\sum_{n = n_{1}}^{n_{2}} n^{-r} \leq \frac{n^{1 - r}_{2} - (n_{1} - 1)^{1 - r}}{1 - r}.
		\end{equation*}		
 		The result remains true for $r > 1$ when $n_{1} \geq 2$.
	\end{lemma}	
	\begin{lemma} \label{L:Technical}
    	For all $k \geq 1$, the following inequality holds true
    	\begin{equation} \label{L:Technical:1}
     		\sum_{s = 0}^{k - 1} \alpha_{s}t_{s} \leq
     		\begin{cases}
				\left(\left(k + a - 1\right)^{2 - \gamma} - 1\right)\left(2 - \gamma\right)^{-1}, & 0 < \gamma < 2, \\
    			\ln\left(k + a - 1 \right), & \gamma = 2, \\
    			\left(1 - \left(k + a - 1\right)^{2 - \gamma}\right)\left(\gamma - 2\right)^{-1}, & \gamma > 2. \\
    		\end{cases}
    	\end{equation}
	\end{lemma}
    \begin{proof}
        Let $k \geq 1$. Using the definitions of $\alpha_{k}$ and $t_{k}$ (see \eqref{Bi-AGUpdate:0} and \eqref{D:Tk}, respectively), we obtain that
        \begin{equation} \label{L:Technical:2}
        	\sum_{s = 0}^{k - 1} \alpha_{s}t_{s} = \frac{1}{a}\sum_{s = 0}^{k - 1} \left(s + a\right)^{1 - \gamma} = \frac{1}{a}\sum_{s = a}^{k + a - 1} s^{1 - \gamma} \leq \frac{1}{a}\sum_{s = 2}^{k + a - 1} s^{1 - \gamma},
        \end{equation}
		where the last inequality follows from the fact that $a \geq 2$ since we sum (potentially) more elements. Now, we split to three cases. First, if $\gamma = 2$, we get from \eqref{L:Technical:2} that
        \begin{equation*}
        	\sum_{s = 0}^{k - 1} \alpha_{s}t_{s} \leq \frac{1}{a}\sum_{s = 2}^{k + a - 1} s^{-1} \leq \frac{1}{a}\ln\left(k + a - 1\right).
        \end{equation*}
		For all $0 < \gamma \neq 2$, we first obtain from \eqref{L:Technical:2} and Lemma \ref{L:TechSum} with $n_{1} = 2$, $n_{2} = k + a - 1$ and $r = \gamma - 1$, that
        \begin{equation} \label{L:Technical:3}
        	\sum_{s = 2}^{k + a - 1} s^{1 - \gamma} \leq \frac{\left(k + a - 1\right)^{2 - \gamma} - 1^{2 - \gamma}}{2 - \gamma} = \frac{\left(k + a - 1\right)^{2 - \gamma} - 1}{2 - \gamma}.
        \end{equation}
        If $0 < \gamma < 2$, the desired result follows immediately from \eqref{L:Technical:3}. If $\gamma > 2$, we obtain from \eqref{L:Technical:3} that
        \begin{equation*}
        	\sum_{s = 0}^{k - 1} \alpha_{s}t_{s} \leq \frac{1}{a} \cdot\frac{\left(k + a - 1\right)^{2 - \gamma} - 1}{2 - \gamma} = \frac{1}{a} \cdot \frac{1 - \left(k + a - 1\right)^{2 - \gamma}}{\gamma - 2}.
        \end{equation*}
        This concludes the proof of the desired result.
    \end{proof}

\subsubsection{Inner Rates of Convergence}	
	We next establish an accelerated convergence rate of order $O(1/k^{2})$ for the inner objective when $\gamma > 2$. While this result mirrors the optimal rate in single-level settings and seems to be the first fast inner rate result, it should be noted that, in the absence of a corresponding outer guarantee, its bi-level relevance is primarily illustrative of the algorithm's behavior in this regime. To the best of our knowledge, achieving simultaneous results in the $\gamma > 2$ regime is an open problem.
	
	\begin{theorem}[Fast inner rate for $\gamma > 2$] \label{T:FastRateIn}
    	Let $\Seq{\bx}{k}$ be a sequence generated by FBi-PG with $\gamma > 2$. For all $k \in \nn$ and $\bx' \in X'$, the following inequality holds true
    	\begin{equation} \label{T:FastRateIn:1}
		  	\varphi\left(\bx^{k}\right) - \varphi\left(\bx'\right) \leq \frac{a^{2}}{2\left(k + 1\right)^{2}}\left(\beta\norm{\bx^{0} - \bx'}^{2} + \frac{2}{\gamma - 2}\left(\omega\left(\bx'\right) - \omega^{\ast}\right)\right).
    	\end{equation}
	\end{theorem}
    \begin{proof}
        From Proposition \ref{P:MainIne} (see \eqref{P:MainIne:0}) we have that
		\begin{equation*}
			t_{k - 1}^{2}\left(\varphi\left(\bx^{k}\right) - \varphi\left(\bx'\right)\right) \leq \frac{\beta}{2}\norm{\bx^{0} - \bx'}^{2} + \left(\omega\left(\bx'\right) - \omega^{\ast}\right)\sum_{s = 0}^{k - 1} \alpha_{s}t_{s}.
		\end{equation*}
		Multiplying both sides by $t_{k - 1}^{-2} > 0$ yields that
    	\begin{equation} \label{T:FastRateIn:2}
     		\varphi\left(\bx^{k}\right) - \varphi\left(\bx'\right) \leq t_{k - 1}^{-2}\frac{\beta}{2}\norm{\bx^{0} - \bx'}^{2} + \left(\omega\left(\bx'\right) - \omega^{\ast}\right)t_{k - 1}^{-2}\sum_{s = 0}^{k - 1} \alpha_{s}t_{s}.
    	\end{equation}
    	Now, from Lemma \ref{L:Technical} with $\gamma > 2$, we have that
   	 	\begin{equation*}
   	 		t_{k - 1}^{-2}\sum_{s = 0}^{k - 1} \alpha_{s}t_{s} \leq \frac{t_{k - 1}^{-2}}{\gamma - 2}\left(1 - \left(k + a - 1\right)^{2 - \gamma}\right) \leq \frac{t_{k - 1}^{-2}}{\gamma - 2},
   	 	\end{equation*}
   	 	which combined with \eqref{T:FastRateIn:2} yields that (recall that $\omega\left(\bx'\right) - \omega^{\ast} \geq 0$ since $\omega^{\ast} = \inf_{\bx \in \real^{n}} \omega\left(\bx\right)$)
    	\begin{equation*}
     		\varphi\left(\bx^{k}\right) - \varphi\left(\bx'\right) \leq t_{k - 1}^{-2}\frac{\beta}{2}\norm{\bx^{0} - \bx'}^{2} + \frac{t_{k - 1}^{-2}}{\gamma - 2}\left(\omega\left(\bx'\right) - \omega^{\ast}\right) = \frac{t_{k - 1}^{-2}}{2}\left(\beta\norm{\bx^{0} - \bx'}^{2} + \frac{2}{\gamma - 2}\left(\omega\left(\bx'\right) - \omega^{\ast}\right)\right).
    	\end{equation*}
   	 	The desired result now follows from the fact that $t_{k - 1}^{-1} = a/\left(k + a - 1\right) \leq a/\left(k + 1\right)$, where the inequality follows from the fact that $a \geq 2$.
    \end{proof}
    While our primary goal was to achieve a fast rate of convergence, the results obtained above also enable us to derive the following rates of convergence concerning the inner problem. These rates are instrumental for deriving below additional results, particularly simultaneous rates pertaining to {\em both} the inner and outer optimization problems.
	\begin{proposition}[Inner rate for $0 < \gamma \leq 2$] \label{P:RateIn}
    	Let $\Seq{\bx}{k}$ be a sequence generated by FBi-PG. For all $k \in \nn$ and $\bx' \in X'$, we have
    	\begin{equation} \label{P:RateIn:1}
		  	\varphi\left(\bx^{k}\right) - \varphi\left(\bx'\right) \leq \frac{a^{2}\beta}{2\left(k + 1\right)^{2}}\norm{\bx^{0} - \bx'}^{2} + \frac{a^{2}c_{k}}{\left(k + 1\right)^{2}}\left(\omega\left(\bx'\right) - \omega^{\ast}\right),
    	\end{equation}
    	where
    	\begin{equation} \label{P:RateIn:2}
    		c_{k} =
    		\begin{cases}
				\left(2 - \gamma\right)^{-1}\left(k + 1\right)^{2 - \gamma}, & 0 < \gamma < 2, \\
    			\ln\left(k + a - 1\right), & \gamma = 2.
    		\end{cases}
    	\end{equation}
	\end{proposition}
    \begin{proof}
        From \eqref{T:FastRateIn:2}, we obtain that
    	\begin{align*}
     		\varphi\left(\bx^{k}\right) - \varphi\left(\bx'\right) & \leq t_{k - 1}^{-2}\frac{\beta}{2}\norm{\bx^{0} - \bx'}^{2} + \left(\omega\left(\bx'\right) - \omega^{\ast}\right)t_{k - 1}^{-2}\sum_{s = 0}^{k - 1} \alpha_{s}t_{s} \\
     		& \leq \frac{a^{2}\beta}{2\left(k + 1\right)^{2}}\norm{\bx^{0} - \bx'}^{2} + \left(\omega\left(\bx'\right) - \omega^{\ast}\right)t_{k - 1}^{-2}\sum_{s = 0}^{k - 1} \alpha_{s}t_{s},
    	\end{align*}
    	where the last inequality follows from the fact that $t_{k - 1}^{-1} = a/\left(k + a - 1\right) \leq a/\left(k + 1\right)$ since we used the fact that $a \geq 2$. Now, to obtain the desired result, we split the proof to two cases. If $\gamma = 2$, we immediately get from Lemma \ref{L:Technical} that
    	 \begin{equation*}
   	 		t_{k - 1}^{-2}\sum_{s = 0}^{k - 1} \alpha_{s}t_{s} \leq t_{k - 1}^{-2}\ln\left(k + a - 1\right) = \frac{a^{2}\ln\left(k + a - 1\right)}{\left(k + a - 1\right)^{2}} \leq \frac{a^{2}\ln\left(k + a - 1\right)}{\left(k + 1\right)^{2}},
   	 	\end{equation*}
   	 	where the last inequality follows from the fact that $a \geq 2$. Moreover, when $0 < \gamma < 2$, we have from Lemma \ref{L:Technical} that
   	 	\begin{equation*}
   	 		t_{k - 1}^{-2}\sum_{s = 0}^{k - 1} \alpha_{s}t_{s} \leq \frac{t_{k - 1}^{-2}\left(k + a - 1\right)^{2 - \gamma} - t_{k - 1}^{-2}}{2 - \gamma} \leq \frac{a^{2}}{\left(2 - \gamma\right)\left(k + a - 1\right)^{\gamma}} \leq \frac{a^{2}}{\left(2 - \gamma\right)\left(k + 1\right)^{\gamma}},
   	 	\end{equation*}
   	 	where the second inequality follows from the definition of $t_{k}$ (see \eqref{D:Tk}) after eliminating the term $t_{k - 1}^{2} \geq 0$ and the last inequality follows from the fact that  $a \geq 2$. This concludes the proof.
    \end{proof}

\subsubsection{Outer Rate of Convergence}	
    Now, we can prove a rate of convergence in terms of the outer optimization problem.
	\begin{theorem}[Outer rate for $0 < \gamma < 2$] \label{T:RateOut}
    	Let $\Seq{\bx}{k}$ be a sequence generated by FBi-PG with $\gamma \in \left(0 , 2\right)$. For all $k \in \nn$ and $\bx' \in X'$, the following inequality holds true
        \begin{equation} \label{T:RateOut:1}
        	\min_{1 \leq s \leq k} \omega\left(\bx^{s}\right) - \omega\left(\bx'\right) \leq \frac{a^{2}\beta}{2\left(k + 1\right)^{2 - \gamma}}\norm{\bx^{0} - \bx'}^{2}.
        \end{equation}
	\end{theorem}
    \begin{proof}
        For the simplicity of the proof, we denote ${\bar \eta}_{k} = 1/\sum_{s = 1}^{k} \eta_{s} > 0$ (recall that $\eta_{k} > 0$ for all $k \in \nn$, see \eqref{deta}). We first prove the result under the assumption that ${\bar \eta}_{k - 1}\sum_{s = 1}^{k - 1} \eta_{s}\omega\left(\bx^{s}\right) - \omega\left(\bx'\right) \geq 0$. Indeed, in this case, we have that
   		\begin{equation*}
   			0 \leq {\bar \eta}_{k - 1}\sum_{s = 1}^{k - 1} \eta_{s}\omega\left(\bx^{s}\right) - \omega\left(\bx'\right) = {\bar \eta}_{k - 1}\sum_{s = 1}^{k - 1} \eta_{s}\left(\omega\left(\bx^{s}\right) - \omega\left(\bx'\right)\right),
   		\end{equation*}
   		and therefore
   		\begin{equation} \label{T:RateOut:2}
   			0 \leq \sum_{s = 1}^{k - 1} \eta_{s}\left(\omega\left(\bx^{s}\right) - \omega\left(\bx'\right)\right) = \sum_{s = 0}^{k - 1} \eta_{s}\left(\omega\left(\bx^{s}\right) - \omega\left(\bx'\right)\right),
   		\end{equation}
   		where the equality follows from the fact that $\eta_{0} = 0$. Using Proposition \ref{P:SimIne} (see \eqref{C:SimIne:0}) combined with \eqref{T:RateOut:2}, we get that
   		\begin{equation*}
   			t_{k - 1}^{2}\left(F_{k - 1}\left(\bx^{k}\right) - F_{k - 1}\left(\bx'\right)\right) \leq \frac{\beta}{2}\norm{\bx^{0} - \bx'}^{2} - \sum_{s = 0}^{k - 1} \eta_{s}\left(\omega\left(\bx^{s}\right) - \omega\left(\bx'\right)\right) \leq \frac{\beta}{2}\norm{\bx^{0} - \bx'}^{2}.
   		\end{equation*}
   		From the definition of $F_{k - 1} \equiv \varphi + \alpha_{k - 1}\omega$ we get that
    	\begin{equation*}
    		t_{k - 1}^{2}\left(\varphi\left(\bx^{k}\right) - \varphi\left(\bx'\right)\right) + \alpha_{k - 1}t_{k - 1}^{2}\left(\omega\left(\bx^{k}\right) - \omega\left(\bx'\right)\right) = t_{k - 1}^{2}\left(F_{k - 1}\left(\bx^{k}\right) - F_{k - 1}\left(\bx'\right)\right) \leq \frac{\beta}{2}\norm{\bx^{0} - \bx'}^{2},
    	\end{equation*}
   		and since $\varphi\left(\bx^{k}\right) - \varphi\left(\bx'\right) \geq 0$ for all $k \in \nn$, after dividing both sides by $t_{k - 1}^{2}\alpha_{k - 1}$ we then obtain that
   		\begin{equation} \label{T:RateOut:3}
   			\omega\left(\bx^{k}\right) - \omega\left(\bx'\right) \leq \alpha_{k -1}^{-1}t_{k - 1}^{-2}\frac{\beta}{2}\norm{\bx^{0} - \bx'}^{2} = \frac{a^{2}\beta}{2\left(k + a - 1\right)^{2 - \gamma}}\norm{\bx^{0} - \bx'}^{2},
   		\end{equation}
		where the last equality follows from the definitions of $\alpha_{k}$ and $t_{k}$ (see \eqref{Bi-AGUpdate:0} and \eqref{D:Tk}, respectively). Recalling that $a \geq 2$ immediately proves \eqref{T:RateOut:1} since obviously $\min_{1 \leq s \leq k} \omega\left(\bx^{s}\right) \leq \omega\left(\bx^{k}\right)$. Since the proof was done under the assumption that ${\bar \eta}_{k - 1}\sum_{s = 1}^{k - 1} \eta_{s}\omega\left(\bx^{s}\right) - \omega\left(\bx'\right) \geq 0$, we note that if
   		\begin{equation} \label{T:RateOut:4}		
			{\bar \eta}_{k - 1}\sum_{s = 1}^{k - 1} \eta_{s}\omega\left(\bx^{s}\right) - \omega\left(\bx'\right) \leq 0,
   		\end{equation}
		then we have that
   		\begin{equation*}
   			\min_{1 \leq s \leq k} \omega\left(\bx^{s}\right) - \omega\left(\bx'\right) \leq \min_{1 \leq s \leq k - 1} \omega\left(\bx^{s}\right) - \omega\left(\bx'\right) \leq {\bar \eta}_{k - 1}\sum_{s = 1}^{k - 1} \eta_{s}\omega\left(\bx^{s}\right) - \omega\left(\bx'\right) \leq 0,
   		\end{equation*}
   		and therefore \eqref{T:RateOut:1} obviously true.
	\end{proof}		

\subsubsection{Simultaneous Rates of Convergence}	

	Now, we prove a simultaneous ergodic rate of convergence for both the inner and outer optimization problems.
	\begin{theorem}[Simultaneous rates for $\gamma = 1$] \label{T:RateSimul}
    	Let $\Seq{\bx}{k}$ be a sequence generated by FBi-PG with $\gamma = 1$. For all $k \in \nn$ and $\bx' \in X'$, the following inequalities hold true
    	\begin{equation} \label{T:RateSimul:1}
		  	\varphi\left({\tilde \bx}^{k}\right) - \varphi\left(\bx'\right) \leq \frac{\pi^{2}a^{2}\beta}{12k}\norm{\bx^{0} - \bx'}^{2} + \frac{a^{2}\ln\left(k + 1\right)}{k}\left(\omega\left(\bx'\right) - \omega^{\ast}\right),
    	\end{equation}
		and
        \begin{equation} \label{T:RateSimul:2}
        	\omega\left({\tilde \bx}^{k}\right) - \omega\left(\bx'\right) \leq \frac{a^{2}\beta}{2\left(k + 1\right)}\norm{\bx^{0} - \bx'}^{2},
        \end{equation}
        where ${\tilde \bx}^{k} = \argmin \left\{ \omega\left(\bx\right) : \, \bx \in \left\{ \bx^{k} , {\bar \bx}^{k} \right\} \right\}$ and ${\bar \bx}^{k} = \sum_{s = 1}^{k} \bx^{s}/k$.
	\end{theorem}	
	\begin{proof}
		In order to prove \eqref{T:RateSimul:1}, for all $ 1 \leq s \leq k$, we get from Proposition \ref{P:RateIn} (see \eqref{P:RateIn:1}) with $k = s$ that (recall that $\gamma = 1$, and hence $c_{s} \equiv s + 1$)
		\begin{equation}
			\varphi\left(\bx^{s}\right) - \varphi\left(\bx'\right) \leq \frac{a^{2}\beta}{2\left(s + 1\right)^{2}}\norm{\bx^{0} - \bx'}^{2} + \frac{a^{2}}{s + 1}\left(\omega\left(\bx'\right) - \omega^{\ast}\right).
		\end{equation}
		Therefore, from the convexity of $\varphi$ we obtain that
		\begin{align*}
   			\varphi\left({\bar \bx}^{k}\right) - \varphi\left(\bx'\right) & \leq \frac{1}{k}\sum_{s = 1}^{k} \left(\varphi\left(\bx^{s}\right) - \varphi\left(\bx'\right)\right) \\
   			& \leq \frac{1}{k}\sum_{s = 1}^{k} \left(\frac{a^{2}\beta}{2\left(s + 1\right)^{2}}\norm{\bx^{0} - \bx'}^{2} + \frac{a^{2}}{s + 1}\left(\omega\left(\bx'\right) - \omega^{\ast}\right)\right) \\
   			& = \frac{1}{k}\left(\frac{a^{2}\beta}{2}\norm{\bx^{0} - \bx'}^{2}\sum_{s = 2}^{k + 1} \frac{1}{s^{2}} + a^{2}\left(\omega\left(\bx'\right) - \omega^{\ast}\right)\sum_{s = 2}^{k + 1} \frac{1}{s}\right).
   		\end{align*}	
   		Now, using classical results we have that $\sum_{s = 2}^{k + 1} s^{-2} < \pi^{2}/6$ and $\sum_{s = 2}^{k + 1} s^{-1} \leq \ln\left(k + 1\right)$, which implies that
		\begin{equation} \label{T:RateSimul:21}
			\varphi\left({\bar \bx}^{k}\right) - \varphi\left(\bx'\right) \leq \frac{\pi^{2}a^{2}\beta}{12k}\norm{\bx^{0} - \bx'}^{2} + \frac{a^{2}\ln\left(k + 1\right)}{k}\left(\omega\left(\bx'\right) - \omega^{\ast}\right).
		\end{equation}	
		Again from Proposition \ref{P:RateIn} (see \eqref{P:RateIn:1}) and recall that $\gamma = 1$ (hence $c_k \equiv k + 1$) we obtain that
		\begin{align}
			\varphi\left(\bx^{k}\right) - \varphi\left(\bx'\right) & \leq \frac{a^{2}\beta}{2\left(k + 1\right)^{2}}\norm{\bx^{0} - \bx'}^{2} + \frac{a^{2}}{k + 1}\left(\omega\left(\bx'\right) - \omega^{\ast}\right) \nonumber \\
			& \leq \frac{\pi^{2}a^{2}\beta}{12k}\norm{\bx^{0} - \bx'}^{2} + \frac{a^{2}\ln\left(k + 1\right)}{k}\left(\omega\left(\bx'\right) - \omega^{\ast}\right),  \label{T:RateSimul:22}
		\end{align}
		where the last inequality follows from the fact that $1/2 < \pi^{2}/12$ and thanks to the convexity of the function $t\ln(t)$ for $t \geq 0$ we have that $k < \left(k + 1\right)\ln\left(k + 1\right)$. This concludes the proof of \eqref{T:RateSimul:1} due to the definition of ${\tilde \bx}^{k}$ by combining \eqref{T:RateSimul:21} and \eqref{T:RateSimul:22}.
\medskip		
		
		Now, in order to prove \eqref{T:RateSimul:2}, since $\gamma = 1$ and with $t_{k} = \left(k + a\right)/a$ (see \eqref{D:Tk}) we get from the definition of $\alpha_{k}$ (see \eqref{Bi-AGUpdate:0}) that
		\begin{equation} \label{T:RateSimul:3}
			\eta_{k} = \alpha_{k - 1}t_{k - 1}^{2} - \alpha_{k}t_{k}^{2} + \alpha_{k}t_{k} = \frac{k + a - 1}{a^{2}} - \frac{k + a}{a^{2}} + \frac{1}{a} = \frac{a - 1}{a^{2}}.
		\end{equation}		
		Hence, ${\bar \eta}_{k} := 1/\sum_{s = 1}^{k} \eta_{k} = a^{2}/\left(k\left(a - 1\right)\right)$ for all $k \in \nn$ and therefore ${\bar \bx}^{k} = \sum_{s = 1}^{k} \bx^{s}/k = {\bar \eta}_{k}\sum_{s = 1}^{k} \eta_{s}\bx^{s}$.
\medskip
		
		Following the proof of Theorem \ref{T:RateOut} (see \eqref{T:RateOut:3} and \eqref{T:RateOut:4}) we easily obtain that
        \begin{equation} \label{T:RateSimul:4}
        	\min \left\{ \omega\left(\bx^{k}\right) , {\bar \eta}_{k}\sum_{s = 1}^{k} \eta_{s}\omega\left(\bx^{s}\right) \right\}  - \omega\left(\bx'\right) \leq \frac{a^{2}\beta}{2\left(k + 1\right)}\norm{\bx^{0} - \bx'}^{2}.
        \end{equation}		
        Moreover, from the convexity of $\omega$, we get that
		\begin{equation*}
   			\omega\left({\bar \bx}^{k}\right) \leq \frac{1}{k}\sum_{s = 1}^{k} \omega\left(\bx^{s}\right) = \frac{a^{2}}{k\left(a - 1\right)}\sum_{s = 1}^{k} \eta_{s}\omega\left(\bx^{s}\right) = {\bar \eta}_{k}\sum_{s = 1}^{k} \eta_{s}\omega\left(\bx^{s}\right) ,
   		\end{equation*}		
   		where the first and last equalities follow from the facts that $\eta_{k} = \left(a - 1\right)/a^{2}$ and ${\bar \eta}_{k} = a^{2}/\left(k\left(a - 1\right)\right)$, respectively. Integrating it into \eqref{T:RateSimul:4} yields that
		\begin{equation*}
   			\min \left\{ \omega\left(\bx^{k}\right) , \omega\left({\bar \bx}^{k}\right) \right\} - \omega\left(\bx'\right) \leq \min \left\{ \omega\left(\bx^{k}\right) , {\bar \eta}_{k}\sum_{s = 1}^{k} \eta_{s}\omega\left(\bx^{s}\right) \right\} - \omega\left(\bx'\right) \leq \frac{a^{2}\beta}{2\left(k + 1\right)}\norm{\bx^{0} - \bx'}^{2}.
   		\end{equation*}	
   		The desired result now follows immediately from the definition of ${\tilde \bx}^{k}$.
	\end{proof}

\subsection{Improved Simultaneous Rates} \label{SSec:ImprovedRate}
	In order to obtain improved simultaneous rates of convergence, we will impose a H\"olderian type error bound assumption on the inner objective function $\varphi$. For more details and examples of functions satisfying such property, we refer the interested readers to \cite{BNPS2015}. We first recall the definition.
	\begin{definition}[H\"olderian type error bound]
		A proper, lower semicontinuous and convex function $h : \real^{n} \rightarrow \erl$ satisfies an H\"olderian error bound with a parameter $\tau > 0$ if for all $\bx \in \real^{n}$
		\begin{equation}
			\tau \cdot \dist\left(\bx , \argmin h\right)^{2} \leq h\left(\bx\right) - \min_{\bz \in \real^{n}} h\left(\bz\right).
		\end{equation}			
	\end{definition}
	Throughout this section, we make the following additional assumption.
	\begin{assumption} \label{A:AssumpB}
    	The following items hold.
    	\begin{itemize}
        	\item[(i)] $\varphi$ satisfies the H\"olderian error bound with a parameter $\tau > 0$.
        	\item[(ii)] $\idom{\omega} \cap X' \neq \emptyset$.
    	\end{itemize}
	\end{assumption}
	It should be noted that Assumption B(ii) is very mild, as any function with a full domain, such as $\norm{\cdot}_{1}$, would satisfy it. Moreover, note that the subdifferential $\partial \omega\left(\bx'\right)$ is not empty for any $\bx' \in \idom{\omega} \cap X'$.
\medskip
	
	Before establishing the promised improved simultaneous rates of convergence, we first need the following simple bound for $\eta_{k}$, $k \in \nn$, which will be used throughout the section.
	\begin{lemma}\label{L:boundeta}
		$1 < \gamma < 2$. For all $k \in \nn$, one has
		\begin{equation} \label{Etaproperty}
        	\eta_{k}  < \frac{1}{2}\left(k + 1\right)^{1 - \gamma}.
        \end{equation}
   	\end{lemma}
	\begin{proof}
		Using the definitions of $\alpha_{k} = \left(k + a\right)^{-\gamma}$ and $t_{k} = \left(k + a\right)/a$ for all $k \in \nn$, it follows that
		\begin{equation*}
        	\eta_{k} = \alpha_{k - 1}t_{k - 1}^{2} - \alpha_{k}\left(t_{k}^{2} - t_{k}\right) = \frac{1}{a^{2}}\left(k + a - 1\right)^{2 - \gamma} - \frac{1}{a^{2}}\left(k + a\right)^{2 - \gamma} + \frac{1}{a}\left(k + a\right)^{1 - \gamma} < \frac{1}{2}\left(k + 1\right)^{1 - \gamma},
        \end{equation*}
		where the last inequality follows from the facts that $\left(k + a - 1\right)^{2 - \gamma} < \left(k + a\right)^{2 - \gamma}$ and $\left(k + a\right)^{1 - \gamma} < \left(k + 1\right)^{1 - \gamma}$ since $1 < \gamma < 2$ and $a \geq 2 > 1$.
	\end{proof}
	In the next theorem, we establish the improved simultaneous convergence rates.
	\begin{theorem}[Simultaneous rates under H\"olderian EB for $1 < \gamma < 2$] \label{T:SimulHolder}
    	Let $\Seq{\bx}{k}$ be a sequence generated by FBi-PG with $1 < \gamma < 2$. For all $k \in \nn$ and $\bx' \in \idom{\omega} \cap X'$, the following inequalities hold true
    	\begin{itemize}
        	\item[(i)] $\varphi\left(\bx^{k}\right) - \varphi\left(\bx'\right) \leq \frac{aC}{\left(k + 1\right)^{2}}$,
            \item[(ii)] $\omega\left(\bx'\right) - \omega\left(\bx^{k}\right) \leq \frac{C}{a^{2}\left(k + 1\right)}$
        	\item[(iii)] $\omega\left(\bx^{k}\right) - \omega\left(\bx'\right) \leq \frac{C}{\left(k + 1\right)^{2 - \gamma}}$,
    	\end{itemize}
    	where $C = a^{2}\max \left\{ \beta\norm{\bx^{0} - \bx'}^{2} + \left(\omega\left(\bx'\right) - \omega^{\ast}\right) , a^{3}\rho^{2}\tau^{-1}\left(\gamma - 1\right)^{-2} \right\}$ and $\rho = \norm{\xi}$ for some $\xi \in \partial \omega\left(\bx'\right)$.
	\end{theorem}
    \begin{proof}
        For the simplicity of the proof, we denote $\Phi_{k} = \varphi\left(\bx^{k}\right) - \varphi\left(\bx'\right)$ for all $k \in \nn$. We will prove the first item using an induction argument on $k \in \nn$. For $k = 1$, we obtain from Proposition \ref{P:MainIne} that (recall that $t_{0} = 1$)
        \begin{equation*}
			\Phi_{1} = t_{0}^{-1}\left(\varphi\left(\bx^{1}\right) - \varphi\left(\bx'\right)\right) \leq \frac{\beta}{2}\norm{\bx^{0} - \bx'}^{2} + \left(\omega\left(\bx'\right) - \omega^{\ast}\right)\alpha_{0}t_{0} \leq \frac{C}{2a^{2}} \leq \frac{aC}{4},
    	\end{equation*}
		where the second inequality follows from the definition of $C$ and the fact that $a_{0} = a^{-\gamma} < 1/2$ since $\gamma > 1$ and $a \geq 2$, and the last inequality also follows from the fact that $a \geq 2$. This proves the induction base with $k = 1$.
\medskip
		
		Let $k \in \nn$. We assume now that item (i) is true for all  $1 \leq s \leq k - 1$ and we will prove it for $s = k$. For all $1 \leq s \leq k - 1$, we have that
        \begin{equation} \label{T:SimuHolderOmegaBound}
        	\xi^{T}\left(\bx' - \bx^{s}\right) = \xi^{T}\left(\bx' - P_{X^{\ast}}\left(\bx^{s}\right)\right) + \xi^{T}\left(P_{X^{\ast}}\left(\bx^{s}\right) - \bx^s\right) \leq \xi^{T}\left(P_{X^{\ast}}\left(\bx^{s}\right) - \bx^s\right),
        \end{equation}
        where $P_{X^{\ast}}$ denotes the orthogonal projection onto the nonempty set $X^{\ast}$. It should be noted that the last inequality follows from the first-order optimality condition of the constrained minimization of $\omega$ over $X^{\ast}$, which reads (recall that $X' = \argmin \left\{\omega\left(\bx\right) : \, \bx \in X^{\ast} \right\}$):
		\begin{equation*}
			\bo \in \xi + \partial \delta_{X^{\ast}}(\bx') = \xi + N_{X^{\ast}}(\bx'),
		\end{equation*}		
        where $\xi \in \partial \omega\left(\bx'\right)$, and hence with the definition of the normal cone $N_{X^{\ast}}(\cdot)$ it follows that $\xi^{T}\left(\bx' - \bu\right) \leq 0$ for all $\bu \in X^{\ast}$. Therefore, since $P_{X^{\ast}}\left(\bx^{s}\right) \in X^{\ast}$, one has $\xi^{T}\left(\bx' - P_{X^{\ast}}\left(\bx^{s}\right)\right) \leq 0$. Now, using the H\"olderian error bound of the function $\varphi$ (see Assumption \ref{A:AssumpB}(i)) yields
        \begin{equation} \label{T:SimulHolder:GradientBound}
        	\xi^{T}\left(P_{X^{\ast}}\left(\bx^{s}\right) - \bx^s\right) \leq \norm{\xi} \cdot \norm{P_{X^{\ast}}\left(\bx^{s}\right) - \bx^s} = \rho\cdot\dist\left(\bx^{s} , X^{\ast}\right) \leq \frac{\rho}{\sqrt{\tau}}\sqrt{\Phi_{s}},
        \end{equation}
        where the equality follows from the facts that $\norm{\xi} = \rho$ and the definition of the distance function. Hence, using the gradient inequality on the convex function $\omega$, we obtain that (recall that $\xi \in \partial \omega\left(\bx'\right)$)
        \begin{equation} \label{T:SimulHolder:0}
            \omega\left(\bx'\right) - \omega\left(\bx^{s}\right) \leq \xi^{T}\left(\bx' - \bx^{s}\right) \leq \xi^{T}\left(P_{X^{\ast}}\left(\bx^{s}\right) - \bx^s\right) \le \frac{\rho}{\sqrt{\tau}}\sqrt{\Phi_{s}},
        \end{equation}
        where the second inequality follows from \eqref{T:SimuHolderOmegaBound} and the last inequality follows from \eqref{T:SimulHolder:GradientBound}. Using the induction assumption on $\Phi_{s}$ for all $1 \leq s \leq k - 1$, we have thus obtained
        \begin{equation} \label{T:SimulHolder:1}
        	\omega\left(\bx'\right) - \omega\left(\bx^{s}\right) \leq \frac{\rho}{\sqrt{\tau}}\sqrt{\Phi_{s}} \leq \frac{\rho\sqrt{a}\sqrt{C}}{\sqrt{\tau}}\left(s + 1\right)^{-1}.
        \end{equation}
       Therefore, using \eqref{Etaproperty} and \eqref{T:SimulHolder:1}, we obtain that (recall that $\eta_{0} = 0$)
        \begin{equation*}
        	\sum_{s = 0}^{k - 1} \eta_{s}\left(\omega\left(\bx'\right) - \omega\left(\bx^{s}\right)\right) = \sum_{s = 1}^{k - 1} \eta_{s}\left(\omega\left(\bx'\right) - \omega\left(\bx^{s}\right)\right) \leq \frac{\rho\sqrt{a}\sqrt{C}}{2\sqrt{\tau}}\sum\limits_{s = 1}^{k - 1} \left(s + 1\right)^{-\gamma} = \frac{\rho\sqrt{a}\sqrt{C}}{2\sqrt{\tau}}\sum\limits_{s = 2}^{k} s^{-\gamma}.
        \end{equation*}
		Using Lemma \ref{L:TechSum} with $n_{1} = 2$, $n_{2} = k$ and $r = \gamma > 1$, we have
        \begin{equation} \label{T:SimulHolder:3}
        	\sum_{s = 0}^{k - 1} \eta_{s}\left(\omega\left(\bx'\right) - \omega\left(\bx^{s}\right)\right) \leq \frac{\rho\sqrt{a}\sqrt{C}}{2\sqrt{\tau}}\sum_{s = 2}^{k} s^{-\gamma} = \frac{\rho\sqrt{a}\sqrt{C}}{2\sqrt{\tau}} \cdot\frac{1^{1 - \gamma} - \left(k - 1\right)^{1 - \gamma}}{\gamma - 1} \leq \frac{\rho\sqrt{a}\sqrt{C}}{2\sqrt{\tau}\left(\gamma - 1\right)}.
        \end{equation}
        From Proposition \ref{P:SimIne} and \eqref{T:SimulHolder:3}, we get that
    	\begin{equation*}
 		  	t_{k - 1}^{2}\left(F_{k - 1}\left(\bx^{k}\right) - F_{k - 1}\left(\bx'\right)\right) \leq \frac{\beta}{2}\norm{\bx^{0} - \bx'}^{2} - \sum_{s = 0}^{k - 1} \eta_{s}\left(\omega\left(\bx^{s}\right) - \omega\left(\bx'\right)\right) \leq \frac{\beta}{2}\norm{\bx^{0} - \bx'}^{2} + \frac{\rho\sqrt{a}\sqrt{C}}{2\sqrt{\tau}\left(\gamma - 1\right)}.
    	\end{equation*}
		Hence, from the definition of $C$, we get
    	\begin{equation} \label{T:SimulHolder:4}
 		  	t_{k - 1}^{2}\left(F_{k - 1}\left(\bx^{k}\right) - F_{k - 1}\left(\bx'\right)\right) \leq \frac{C}{2a^{2}} + \frac{\sqrt{C}}{2} \cdot \frac{\sqrt{C}}{a^{2}} = \frac{C}{a^{2}}.
    	\end{equation}
    	Moreover, using \eqref{T:SimulHolder:0} with $s = k$, we have that
    	\begin{equation*}
    		\omega\left(\bx^{k}\right) - \omega\left(\bx'\right) \geq -\frac{\rho}{\sqrt{\tau}}\sqrt{\Phi_{k}}.
    	\end{equation*}
    	Using this, the definition of $F_{k - 1} \equiv \varphi + \alpha_{k - 1}\omega$ and \eqref{T:SimulHolder:4} yields that
    	\begin{align}
 		  	\Phi_{k} - \frac{\rho\alpha_{k - 1}}{\sqrt{\tau}}\sqrt{\Phi_{k}} & \leq \Phi_{k} + \alpha_{k - 1}\left(\omega\left(\bx^{k}\right) - \omega\left(\bx'\right)\right) = F_{k - 1}\left(\bx^{k}\right) - F_{k - 1}\left(\bx'\right) \leq \frac{Ct_{k - 1}^{-2}}{a^{2}} \leq \frac{C}{\left(k + 1\right)^{2}}, \label{T:SimulHolder:5}
    	\end{align}
    	where the last inequality follows from the fact that $t_{k - 1}^{-1} = a/\left(k + a - 1\right) \leq a/\left(k + 1\right)$ since $a \geq 2$. Now, we will prove the induction step by splitting the proof into two cases. First, we assume that
    	\begin{equation*}
 		  	\frac{\Phi_{k}}{a} \leq \Phi_{k} - \frac{\rho\alpha_{k - 1}}{\sqrt{\tau}}\sqrt{\Phi_{k}}.
    	\end{equation*}
    	Using \eqref{T:SimulHolder:5} yields that
    	\begin{equation*}
 		  	\frac{\Phi_{k}}{a} \leq \Phi_{k} - \frac{\rho\alpha_{k - 1}}{\sqrt{\tau}}\sqrt{\Phi_{k}} \leq \frac{C}{\left(k + 1\right)^{2}},
    	\end{equation*}
    	and therefore we obviously have that $\Phi_{k} \leq aC/\left(k +  1\right)^{2}$, as required. Second, we assume the converse, that is
    	\begin{equation*}
 		  	\Phi_{k} - \frac{\rho\alpha_{k - 1}}{\sqrt{\tau}}\sqrt{\Phi_{k}} < \frac{\Phi_{k}}{a}.
    	\end{equation*}
    	Therefore, by rearranging the inequality we get that $\Phi_{k} (a-1)/a < \left(\rho\alpha_{k - 1}/\sqrt{\tau}\right)\sqrt{\Phi_{k}}$. Dividing both sides by $\left(a - 1\right)\sqrt{\Phi_{k}}/a$ and then squaring both sides yields
    	\begin{equation*}
    		\Phi_{k} < \frac{a^{2}\rho^{2}}{\tau(a-1)^{2}}\alpha_{k - 1}^{2} \leq \frac{a{^2}\rho^{2}}{\tau(a-1)^{2}}\left(k + 1\right)^{-2} \le \frac{a^{2}\rho^{2}}{\tau\left(\gamma - 1\right)^{2}}\left(k + 1\right)^{-2} \leq \frac{C}{a^{3}\left(k + 1\right)^{2}},
    	\end{equation*}
    	where the second inequality follows from the definition of $\alpha_{k - 1} = \left(k + a - 1\right)^{-\gamma} < \left(k + 1\right)^{-1}$ since $\gamma > 1$ and $a \geq 2$, the third inequality follows from the fact that $0 < \gamma - 1 < a - 1$ since $a \geq 2$ and $1 < \gamma < 2$ while the last inequality follows again from the definition of $C$. Thus, we obtain that $\Phi_{k} \leq aC/\left(k + 1\right)^{2}$ since $a \geq 2$, which proves the desired result.
\medskip
    	
		Now, we will prove the second item. Using \eqref{T:SimulHolder:1} with $s = k$, yields
        \begin{equation*}
        	\omega\left(\bx'\right) - \omega\left(\bx^{k}\right) \leq \frac{\rho\sqrt{a}\sqrt{C}}{\sqrt{\tau}}\left(k + 1\right)^{-1} \leq \frac{\rho\sqrt{a}\sqrt{C}}{\sqrt{\tau}\left(\gamma - 1\right)}\left(k + 1\right)^{-1} \leq \frac{C}{a^{2}}\left(k + 1\right)^{-1},
        \end{equation*}
		where the second inequality follows from the fact that $1 < \left(\gamma - 1\right)^{-1}$ since $1 < \gamma < 2$ and the last inequality follows from the definition of $C$. For proving the last item, using \eqref{T:SimulHolder:4}, we have that
    	\begin{equation*}
 		  	\alpha_{k - 1}\left(\omega\left(\bx^{k}\right) - \omega\left(\bx'\right)\right) \leq \Phi_{k} + \alpha_{k - 1}\left(\omega\left(\bx^{k}\right) - \omega\left(\bx'\right)\right) = F_{k - 1}\left(\bx^{k}\right) - F_{k - 1}\left(\bx'\right) \leq \frac{Ct_{k - 1}^{-2}}{a^{2}},
    	\end{equation*}
    	where the first inequality follows from the fact that $\Phi_{k} \geq 0$ and the equality follows from the definition of $F_{k - 1} \equiv \varphi + \alpha_{k - 1}\omega$. Hence, immediately from the definitions of $\alpha_{k}$ and $t_{k}$ (see \eqref{Bi-AGUpdate:0} and \eqref{D:Tk}, respectively) we obtain that
    	\begin{equation*}
 		  	\omega\left(\bx^{k}\right) - \omega\left(\bx'\right) \leq \frac{Ct_{k - 1}^{-2}}{a^{2}\alpha_{k - 1}} = \frac{C}{\left(k + a - 1\right)^{2 - \gamma}} \leq \frac{C}{\left(k + 1\right)^{2 - \gamma}},
    	\end{equation*}
        where the last inequality follows from the fact that $a - 1 \geq 1$ since $a \geq 2$. This proves the desired result.
	\end{proof}
	Before concluding this section, we note that, as explained in Remark \ref{R:Lift}, applying FBi-PG to the lifted problem avoids the need to compute the proximal mapping of the sum of two functions. In the following remark, we record the  effect on the convergence rates using the lifted problem.
	\begin{remark}
		Following Remark \ref{R:Lift}, applying any of our convergence rate results to the lifted inner objective $\bar{\varphi}$ yields the bound ${\bar \varphi}\left(\bw^{k}\right) - {\bar \varphi}\left(\bw'\right) \leq {\cal R}_{k}^{I}$, where $\mathcal{R}_k^I$ denotes any of the upper bound (which depends of the choice of $\gamma$) as derived in Theorems~\ref{T:FastRateIn}, \ref{T:RateSimul}, and Proposition~\ref{P:RateIn}. As a consequence, recalling that ${\bar \varphi}\left(\bw'\right) = \varphi\left(\bx' \right)$, it is immediate to deduce that
				\begin{equation*}
					\varphi\left(\bx^{k}\right)- \varphi\left(\bx'\right) \leq {\cal R}_{k}^{I} \quad \text{and} \quad \norm{\bx^{k} - \bz^{k}} \leq \sqrt{{\cal R}_{k}^{I}}.
    			\end{equation*}
				This shows that the lifting approach allows us to completely circumvent the need for computing the proximal mapping of the composite term $\sigma + \psi$, while still retaining the original convergence rates for the inner problem. The situation for the outer objective is more subtle. While at optimality we have $\bar{\omega}\left(\bw’\right) = \omega\left(\bx’\right)$ (since $\bx’ = \bz’$), this equality does not hold in general. This can be circumvented by assuming that $\psi$ is $L_\psi$-Lipschitz continuous. Applying our convergence rate results to the lifted objective $\bar{\omega}$ yields the bound ${\bar \omega}\left(\bw^{k}\right) - {\bar \omega}\left(\bw'\right) \leq {\cal R}_{k}^{O}$, where $\mathcal{R}_k^O$ depends on $\gamma$ as established in Theorems~\ref{T:RateOut} and \ref{T:RateSimul}. It is then easy to show that the outer rate for the original outer objective function satisfies:
				\begin{equation*}
					\omega\left(\bx^{k}\right) - \omega\left(\bx'\right) \leq {\cal R}_{k}^{O} + L_{\psi}\norm{\bx^{k} - \bz^{k}} \leq {\cal R}_{k}^{O} + L_{\psi}\sqrt{{\cal R}_{k}^{I}}.
				\end{equation*}
In summary, by adopting the lifting approach, we eliminate the computational bottleneck associated with evaluating the proximal mapping of a composite term, while preserving the inner convergence rates derived in our analysis. For the outer problem, we must pay a price, which is expressed in terms of an additional Lipschitz continuity assumption, and a  potentially weaker rate, due to the added Lipschitz error term.
	\end{remark}
	
\section{Pointwise Convergence Analysis} \label{Sec:ConvAnal}
	This section is devoted for proving that the sequence $\Seq{\bx}{k}$ generated by FBi-PG converges to an optimal solution of the bi-level optimization problem \eqref{Prob:OP}. This result is achieved under the Assumption \ref{A:AssumpB}, meaning in the setting where the inner objective function $\varphi\left(\cdot\right)$ satisfies the H\"olderian error bound property. Moreover, throughout this section we assume that the sequences $\seq{\alpha}{k}$ and $\seq{t}{k}$ are defined with $1 < \gamma < 2$ and $a > 2$.
\medskip

	As already mentioned, each iteration of FBi-PG can be seen as an iteration of FISTA applied to the regularized function $F_{k}$. Unfortunately, due to the dynamic nature of the regularizing sequence $\seq{\alpha}{k}$, much like we could not apply the convergence rate results of FISTA to get corresponding fast rate for FBi-PG, we also cannot directly apply the existing pointwise convergence results of FISTA established in \cite{CD2015} to derive the corresponding convergence of the sequence $\Seq{\bx}{k}$ generated by FBi-PG to an optimal solution of the bi-level optimization problem \eqref{Prob:OP}. Nevertheless, inspired by the proof techniques developed in \cite{CD2015}, below we show that within a similar line of proof,  we can adapt and extend  their result for handling the bi-level setting.
\medskip

	Throughout the rest of this section, we will define for all $k \in \nn$ and $\bx' \in X'$ the following two quantities
	\begin{equation} \label{DefDeltaK}
    	\delta_{k} = \frac{1}{2}\norm{\bx^{k} - \bx^{k - 1}}^{2} \quad \text{and} \quad \mu_{k} = \frac{1}{2}\norm{\bx^{k} - \bx'}^{2}.
	\end{equation}
	The main tasks of the forthcoming analysis will be to establish that both sequences $\seq{\delta}{k}$ and $\seq{\mu}{k}$ converge, from which the culminating main convergence result will follow by standard arguments.
\medskip

	We begin with a technical lemma (see appendix for the proof) that will be useful in proving the main result.
	\begin{lemma} \label{L:SumDelta}
    	Let $\Seq{\bx}{k}$ be a sequence generated by FBi-PG with $1 < \gamma < 2$ and $a > 2$. Then, there exists $M > 0$ such that $\sum_{s = 0}^{k - 1} t_{s - 1}\delta_{s} < M$ for all $k \in \nn$.
	\end{lemma}
	For all $s \in \nn$, we define $\pi_{s,k} = 1$ for $k < s$ and otherwise we define
	\begin{equation}\label{DefBeta}
     	\pi_{s,k} := \prod_{j = s}^{k} \lambda_{j},
	\end{equation}
	where $\lambda_{k} := t_{k}^{-1}\left(t_{k - 1} - 1\right)$ for all $k \in \nn$.
\medskip

	The following bound is based on a result obtained in \cite{CD2015} with a minor modification to our setting.
	\begin{lemma}\label{L:SumTechnical}
		For all $s \geq 1$ we have that
		\begin{equation*}
			\sum_{k = 0}^{\infty} \pi_{s,k} \leq \frac{5a}{2}t_{s - 1}.
		\end{equation*}
	\end{lemma}
	\begin{proof}
		Let $s \geq 1$. Since $\pi_{s,k} = 1$ for all $k < s$ we obtain that
		\begin{equation*}
			\sum_{k = 0}^{\infty} \pi_{s,k} = \sum_{k = 0}^{s - 1} \pi_{s,k} + \sum_{k = s}^{\infty} \pi_{s,k} = \sum_{k = 0}^{s - 1} 1 + \sum_{k = s}^{\infty} \pi_{s,k} = s + \sum_{k = s}^{\infty} \pi_{s,k} \leq s + \frac{s + 5}{2} = \frac{3s + 5}{2},
		\end{equation*}
		where the inequality follows from \cite{CD2015}. Now, since $a - 1 > 1$ we have that $3s + 5 < 5\left(s + a - 1\right) = 5at_{s - 1}$ where the last equality follows from the definition of $t_{s}$ (see \eqref{D:Tk}). This completes the proof.
	\end{proof}
 	The next theorem proves that the FBi-PG's sequence, $\Seq{\bx}{k}$, is convergent to an optimal solution of the bi-level problem \eqref{Prob:OP}.
	\begin{theorem}[Pointwise convergence under H\"olderian EB for $1 < \gamma < 2$] \label{T:Convergence}
		Let $\Seq{\bx}{k}$ be a sequence generated by FBi-PG with $1 < \gamma < 2$ and $a > 2$. Then, the sequence $\Seq{\bx}{k}$ converges to some $\bx' \in X'$.
	\end{theorem}
    \begin{proof}
		Let $\bx' \in X'$. Since $X' \subset X^{\ast}$ it follows that $\bo \in \partial \varphi\left(\bx'\right) = \nabla f\left(\bx'\right) + \partial g\left(\bx'\right)$. Therefore, $-\nabla f\left(\bx'\right) \in \partial g\left(\bx'\right)$. Taking $\bv' \in \partial \psi\left(\bx'\right)$, it follows that
		\begin{equation}
			-\nabla f\left(\bx'\right) + \alpha_{k}\bv' \in \partial g\left(\bx'\right) + \alpha_{k}\partial \psi\left(\bx'\right) \subset \partial \left(g + \alpha_{k}\psi\right)\left(\bx'\right) = \partial g_{k}\left(\bx'\right),
		\end{equation}
		where the inclusion follows from standard convex subdifferential calculus \cite{BC2011-B} and the last equality follows from the definition of $g_{k} \equiv g + \alpha_{k}\psi$. Writing the optimality condition of the proximal step \eqref{Bi-AGUpdate:2} yields that
		\begin{equation*}
			\beta\left(\by^{k} - \bx^{k + 1}\right) - \nabla f_{k}\left(\by^{k}\right) \in \partial g_{k}\left(\bx^{k + 1}\right).
		\end{equation*}
		Now, since the function $g_{k}$ is convex it follows that its subdifferential $\partial g_{k}$ is a monotone mapping \cite{BC2011-B}. Therefore, we have that
		\begin{align}
			0 & \leq \act{\beta\left(\by^{k} - \bx^{k + 1}\right) - \nabla f_{k}\left(\by^{k}\right) - \left(-\nabla f\left(\bx'\right) + \alpha_{k}\bv'\right) , \bx^{k + 1} - \bx'} \nonumber \\
			& = \act{\beta\left(\by^{k} - \bx^{k + 1}\right) - \nabla f_{k}\left(\by^{k}\right) + \nabla f_{k}\left(\bx'\right) - \alpha_{k}\left(\nabla \sigma\left(\bx'\right) + \bv'\right) , \bx^{k + 1} - \bx'} \nonumber \\
			& = \beta\act{\by^{k} - \bx^{k + 1} , \bx^{k + 1} - \bx'} + \act{\nabla f_{k}\left(\bx'\right) - \nabla f_{k}\left(\by^{k}\right) , \bx^{k + 1} - \bx'} + \alpha_{k}\act{\nabla \sigma\left(\bx'\right) + \bv' , \bx' - \bx^{k + 1}}, \label{T:Convergence:0}
		\end{align}
		where the first equality follows from the fact that $f_{k} \equiv f + \alpha_{k}\sigma$. Since $f_{k}$ is convex and has a Lipschitz continuous gradient with a parameter $\beta$ (see Remark \ref{R:Lipfk}), the later is also $(1/\beta)$-co-coercive mapping and hence we have
		\begin{align}
 			\act{\nabla f_{k}\left(\bx'\right) - \nabla f_{k}\left(\by^{k}\right) , \bx^{k + 1} - \bx'} & = \act{\nabla f_{k}\left(\bx'\right) - \nabla f_{k}\left(\by^{k}\right) , \bx^{k + 1} - \by^{k} + \by^{k} - \bx'} \nonumber \\
& \leq \act{\nabla f_{k}\left(\bx'\right) - \nabla f_{k}\left(\by^{k}\right) , \bx^{k + 1} - \by^{k}} - \frac{1}{\beta}\norm{\nabla f_{k}\left(\bx'\right) - \nabla f_{k}\left(\by^{k}\right)}^{2} \nonumber \\
& \leq \frac{1}{\beta}\norm{\nabla f_{k}\left(\bx'\right) - \nabla f_{k}\left(\by^{k}\right)}^{2} + \frac{\beta}{4}\norm{\bx^{k + 1} - \by^{k}}^{2} \nonumber \\
& - \frac{1}{\beta}\norm{\nabla f_{k}\left(\bx'\right) - \nabla f_{k}\left(\by^{k}\right)}^{2} \nonumber \\
& = \frac{\beta}{4}\norm{\bx^{k + 1} - \by^{k}}^{2}, \label{T:Convergence:1}	
		\end{align}
		where the last inequality follows from the fact that $\bu^{T}\bw \leq \norm{\bu}^{2}/\beta + \beta\norm{\bw}^{2}/4$ for all $\bu , \bw \in \real^{n}$ and $\beta > 0$. Combining \eqref{T:Convergence:0} and \eqref{T:Convergence:1} yields after dividing both sides by $\beta$ that
		\begin{equation} \label{T:Convergence:2}
			0 \leq \act{\by^{k} - \bx^{k + 1} , \bx^{k + 1} - \bx'} + \frac{1}{4}\norm{\bx^{k + 1} - \by^{k}}^{2} + \xi^{k + 1},
		\end{equation}
		where $\xi^{k} \equiv \alpha_{k - 1}\act{\nabla \sigma\left(\bx'\right) + \bv' , \bx' -  \bx^{k}}/\beta$ for all $k \in \nn$. Moreover, using the fundamental Pythagoras identity we obtain that
		\begin{align}
			\act{\by^{k} - \bx^{k + 1} , \bx^{k + 1} - \bx'} & = \act{\by^{k} - \bx^{k} , \bx^{k + 1} - \bx'} + \act{\bx^{k} - \bx^{k + 1} , \bx^{k + 1} - \bx'} \nonumber \\
			& = \act{\by^{k} - \bx^{k} , \bx^{k + 1} - \bx'} + \frac{1}{2}\left(\norm{\bx^{k} - \bx'}^{2} - \norm{\bx^{k + 1} - \bx'}^{2} - \norm{\bx^{k + 1} - \bx^{k}}^{2}\right) \nonumber \\
			& = \act{\by^{k} - \bx^{k} , \bx^{k + 1} - \bx'} + \mu_{k} - \mu_{k + 1} - \delta_{k + 1},  \label{T:Convergence:3}	
		\end{align}
		where the third equality follows from the definitions of $\delta_{k}$ and $\mu_{k}$ (see \eqref{DefDeltaK}). Since $\lambda_{k} = t_{k}^{-1}\left(t_{k - 1} - 1\right)$ we can write from \eqref{Bi-AGUpdate:1} that $\by^{k} - \bx^{k} = \lambda_{k}\left(\bx^{k } - \bx^{k - 1}\right)$. Using again the fundamental Pythagoras identity we obtain that
		\begin{align}
			\act{\by^{k} - \bx^{k} , \bx^{k + 1} - \bx'} & = \act{\by^{k} - \bx^{k} , \bx^{k + 1} - \bx^{k}} + \act{\by^{k} - \bx^{k} , \bx^{k} - \bx'} \nonumber \\
			& = \act{\by^{k} - \bx^{k} , \bx^{k + 1} - \bx^{k}} + \lambda_{k}\act{\bx^{k} - \bx^{k - 1} , \bx^{k} - \bx'} \nonumber \\
			& = \act{\by^{k} - \bx^{k} , \bx^{k + 1} - \bx^{k}} + \lambda_{k}\left(\mu_{k} - \mu_{k - 1} + \delta_{k}\right). \label{T:Convergence:4}	
		\end{align}
		Another application of the Pythagoras identity yields that
		\begin{align}
			\act{\by^{k} - \bx^{k} , \bx^{k + 1} - \bx^{k}} & = \frac{1}{2}\left(\norm{\by^{k} - \bx^{k}}^{2} - \norm{\bx^{k + 1} - \by^{k}}^{2} + \norm{\bx^{k + 1} - \bx^{k}}^{2}\right) \nonumber \\
			& = \lambda_{k}^{2}\delta_{k} - \frac{1}{2}\norm{\bx^{k + 1} - \by^{k}}^{2} + \delta_{k + 1}, \label{T:Convergence:5}
		\end{align}
		where the last equality uses the fact that $\by^{k} - \bx^{k} = \lambda_{k}\left(\bx^{k} - \bx^{k - 1}\right)$ and the definition of $\delta_{k}$. Combining \eqref{T:Convergence:2}, \eqref{T:Convergence:3}, \eqref{T:Convergence:4} and \eqref{T:Convergence:5} yields after eliminating $-\norm{\bx^{k + 1} - \by^{k}}^{2}/4 \leq 0$ that
		\begin{equation} \label{T:Convergence:6}
			\mu_{k + 1} - \mu_{k} \leq \lambda_{k}\left(\mu_{k} - \mu_{k - 1} + \delta_{k}\right) + \lambda_{k}^{2}\delta_{k} + \xi^{k + 1} = \lambda_{k}\left(\mu_{k} - \mu_{k - 1}\right) + \left(\lambda_{k}^{2} + \lambda_{k}\right)\delta_{k} + \xi^{k + 1}.
		\end{equation}
		Using the fact that $\lambda_{k} + \lambda_{k}^{2} \leq 2\lambda_{k}$ since $\lambda_{k}^{2} \leq \lambda_{k}$ (recall that $\lambda_{k} \leq 1$ for all $k \in \nn$), we obtain
		\begin{equation} \label{T:Convergence:7}
			\mu_{k + 1} - \mu_{k} \leq \lambda_{k}\left(\mu_{k} - \mu_{k - 1}\right) + 2\lambda_{k}\delta_{k} + \xi^{k + 1}.
		\end{equation}
		By defining $\theta_{k} = [\mu_{k} - \mu_{k - 1}]_{+} := \max \{ \bo , \mu_{k} - \mu_{k - 1} \}$ for all $k \geq 1$, it follows from \eqref{T:Convergence:7} that
		\begin{align}
			\theta_{k + 1} & = [\mu_{k + 1} - \mu_{k}]_{+} \nonumber \\
			& \leq \left[\lambda_{k}\left(\mu_{k} - \mu_{k - 1}\right) + 2\lambda_{k}\delta_{k} + \xi^{k + 1}\right]_{+} \nonumber \\
			& \leq \lambda_{k}\theta_{k} + 2\lambda_{k}[\delta_{k}]_{+} + \left[\xi^{k + 1}\right]_{+} \nonumber \\
			& = \lambda_{k}\left(\theta_{k} + 2\delta_{k}\right) + \left[\xi^{k + 1}\right]_{+}, \label{T:Convergence:8}
		\end{align}
		where the first inequality follows from the fact that $[x]_{+} \leq [y]_{+}$ when $x \leq y$, the second inequality follows from the fact that $[x + y]_{+} \leq [x]_{+} + [y]_{+}$, and the equality follows from the fact that $\lambda_{k} , \delta_{k} \geq 0$. Thus, applying \eqref{T:Convergence:8} successively yields, for all $k \geq 2$, that (we use the convention that $\prod_{j = k + 1}^{k} \lambda_{j} = 1$)
		\begin{align}
			\theta_{k + 1} & \leq \lambda_{k}\theta_{k} + 2\lambda_{k}\delta_{k} + \left[\xi^{k + 1}\right]_{+} \nonumber \\
			& \leq \lambda_{k}\left(\lambda_{k - 1}(\theta_{k - 1} + 2\delta_{k - 1}) + \left[\xi^{k}\right]_{+}\right) + 2\lambda_{k}\delta_{k} + \left[\xi^{k + 1}\right]_{+} \nonumber \\	
			& = \lambda_{k}\lambda_{k - 1}\theta_{k - 1} + 2\lambda_{k}\lambda_{k - 1}\delta_{k - 1} + 2\lambda_{k}\delta_{k} + \lambda_{k}\left[\xi^{k}\right]_{+} + \left[\xi^{k + 1}\right]_{+} \nonumber \\			
			& = \prod_{j = 2}^{k}\lambda_{j}\theta_{2} + 2\sum_{s = 2}^{k}\prod_{j = s}^{k} \lambda_{j}\delta_{s} + \sum_{s = 3}^{k + 1}\prod_{j = s}^{k} \lambda_{j}\left[\xi^{s}\right]_{+} \nonumber \\			
			& \leq 2\sum_{s = 2}^{k}\prod_{j = s}^{k} \lambda_{j}\delta_{s} + \sum_{s = 2}^{k + 1}\prod_{j = s}^{k} \lambda_{j}\left[\xi^{s}\right]_{+} \nonumber \\
			& = 2\sum_{s = 2}^{k} \pi_{s,k}\delta_{s} + \sum_{s = 2}^{k + 1} \pi_{s,k}\left[\xi^{s}\right]_{+},	 \label{T:Convergence:9}
		\end{align}
		where the last inequality follows from \eqref{T:Convergence:8} with $k = 1$ which implies that $\theta_{2} \leq \left[\xi^{2}\right]_{+}$ since $\lambda_{1} = 0$ (recall $\lambda_{1} = t_{1}^{-1}\left(t_{0} - 1\right) = 0$) and the last equality follows from the definition of $\pi_{s,k}$ (see \eqref{DefBeta}). Note that the inequality \eqref{T:Convergence:9} also holds true for $k = 1$ under the convention that $\sum_{s = 2}^{1} \pi_{s,k}\delta_{s} = 0$ (recall that $\pi_{2,1} = 1$). Therefore, summing \eqref{T:Convergence:9} from $k = 0$ to $\infty$, we have
		\begin{align}
			\sum_{k = 0}^{\infty} \theta_{k + 1} & \leq 2\sum_{k = 0}^{\infty}\sum_{s = 2}^{k} \pi_{s,k}\delta_{s} + \sum_{k = 0}^{\infty}\sum_{s = 2}^{k + 1} \pi_{s,k}\left[\xi^{s}\right]_{+} \nonumber \\
			& \leq 2\sum_{k = 0}^{\infty}\sum_{s = 2}^{\infty} \pi_{s,k}\delta_{s} + \sum_{k = 0}^{\infty}\sum_{s = 2}^{\infty} \pi_{s,k}\left[\xi^{s}\right]_{+} \nonumber \\
			& = 2\sum_{s = 2}^{\infty} \left(\sum_{k = 0}^{\infty} \pi_{s,k}\right)\delta_{s} + \sum_{s = 2}^{\infty} \left(\sum_{k = 0}^{\infty} \pi_{s,k}\right)\left[\xi^{s}\right]_{+} \nonumber \\
			& \leq 5a\sum_{s = 2}^{\infty} t_{s - 1}\delta_{s} + \frac{5a}{2}\sum_{s = 2}^{\infty} t_{s - 1}\left[\xi^{s}\right]_{+}, \label{T:Convergence:10}
		\end{align}
		where the last inequality follows from Lemma \ref{L:SumTechnical}. Next, we show that $\seq{\theta}{k}$ is summable. This goal will be achieved by showing that each of two terms on the right-hand side of \eqref{T:Convergence:9} is finite. From Lemma \ref{L:SumDelta} there exists $M_{1} > 0$ such that
		\begin{equation*}
    		\sum_{s = 2}^{\infty} t_{s -1}\delta_{s} \leq \sum_{s = 0}^{\infty} t_{s -1}\delta_{s} \leq M_{1},
		\end{equation*}
		proving the finiteness of the first term. Now, from the definition of $\xi^{k}$ and since $\nabla \sigma\left(\bx'\right) + \bv' \in \partial \omega\left(\bx'\right)$ (recall that $\omega \equiv \sigma + \psi$ and $\bv' \in \partial \psi\left(\bx'\right)$), we have that
		\begin{equation*}
			\xi^{k} = \frac{\alpha_{k - 1}}{\beta}\act{\nabla \sigma\left(\bx'\right) + \bv' , \bx' -  \bx^{k}} = \frac{\alpha_{k - 1}}{\beta}\act{\xi , \bx' -  \bx^{k}},
		\end{equation*}
		with $\xi \in \partial \omega\left(\bx'\right)$, and hence we obtain that
		\begin{equation} \label{T:Convergence:11}
    		\xi^{k} = \frac{\alpha_{k - 1}}{\beta}\act{\xi , \bx' -  \bx^{k}} \leq \frac{\rho\alpha_{k - 1}}{\beta\sqrt{\tau}}\sqrt{\varphi\left(\bx^{k}\right) - \varphi\left(\bx'\right)} \leq \frac{\rho\alpha_{k - 1}\sqrt{aC}}{\beta\sqrt{\tau}\left(k + 1\right)} = \frac{\rho\sqrt{aC}}{\beta\sqrt{\tau}\left(k + 1\right)^{1 + \gamma}},
		\end{equation}
		where the first inequality follows from \eqref{T:SimuHolderOmegaBound} combined with \eqref{T:SimulHolder:GradientBound}, the second inequality follows from Theorem \ref{T:SimulHolder}(i) and the last equality follows from the definition of $\alpha_{k} := \left(k + a\right)^{-\gamma}$ (see \eqref{Alphak}) and the fact that $\left(k + a - 1\right)^{-\gamma} < \left(k + 1\right)^{-\gamma}$ since $a > 2$. Therefore, from \eqref{T:Convergence:11}, we deduce that
		\begin{equation}
    		\sum_{s = 1}^{\infty} t_{s - 1}\left[\xi^{s}\right]_{+} \leq \frac{\rho\sqrt{aC}}{\beta\sqrt{\tau}}\sum_{s = 1}^{\infty} \frac{t_{s - 1}}{\left(s + 1\right)^{1 + \gamma}} \leq \frac{\rho\sqrt{aC}}{2\beta\sqrt{\tau}}\sum_{s = 1}^{\infty} \frac{1}{\left(s + 1\right)^{\gamma}} = \frac{\rho\sqrt{aC}}{2\beta\sqrt{\tau}}\sum_{s = 2}^{\infty} \frac{1}{s^{\gamma}} = \frac{\rho\sqrt{aC}}{2\beta\sqrt{\tau}} \cdot \frac{1}{\gamma - 1},
		\end{equation}
		where the second inequality follows from the fact that $t_{s - 1} = \left(s - 1\right)/a + 1 \leq \left(s - 1\right)/2 + 1 = \left(s + 1\right)/2$ (recall that $a > 2$) and the last inequality follows from Lemma \ref{L:TechSum} with $n_{1} = 2$, $n_{2} = \infty$ and $r = \gamma > 1$. This proves the finiteness of the second term in \eqref{T:Convergence:10} and therefore completes the proof that $\seq{\theta}{k}$ is summable.
		
		Now, we define $a_{k} = \mu_{k} - \sum_{s = 1}^{k} \theta_{s}$ and since $\mu_{k} \geq 0$ for all $k \in \nn$ and $\seq{\theta}{k}$ is summable, we see that $\seq{a}{k}$ is bounded from below. On the other hand, since $\theta_{k} = [\mu_{k} - \mu_{k - 1}]_{+}$, we get that
		\begin{equation*}
			a_{k + 1} = \mu_{k + 1} - \theta_{k + 1} - \sum_{s = 1}^{k} \theta_{s} \leq \mu_{k + 1} - \mu_{k + 1} + \mu_{k} - \sum_{s = 1}^{k} \theta_{s} = a_{k},
		\end{equation*}
		and hence $\seq{a}{k}$ is also a non-increasing sequence and thus is convergent. Since $\mu_{k} = a_{k} + \sum_{s = 1}^{k} \theta_{s}$ and both sequences converge, it follows that $\seq{\mu}{k}$ converges. To complete the proof we will prove that the limit is zero. Since $\seq{\mu}{k}$ converges it is also bounded, and therefore from the definition of $\mu_{k}$, $k \in \nn$, we also obtain that $\Seq{\bx}{k}$ is bounded. Hence, it has at least one converging sub-sequence, which we denote by $\left\{ \bx^{j_{k}} \right\}_{k \in \nn}$ and its limit point with $\tilde{\bx}$. From the fact that $\varphi$ is lower semi-continuous (see Assumption \ref{A:AssumptionA})) and using Theorem \ref{T:SimulHolder}(i), it follows that
		\begin{equation*}
    		\varphi\left(\tilde{\bx}\right) \leq \liminf_{k \rightarrow \infty} \varphi\left(\bx^{j_{k}}\right) = \varphi\left(\bx'\right).
		\end{equation*}
		Similarly, we obtain that $\omega\left(\tilde{\bx}\right) \leq \omega\left(\bx'\right)$. Since $\bx' \in X'$, these two facts yield that $\tilde{\bx} \in X'$. On the other hand, as we proved above, the sequence $\{\norm{\bx^{k} - \bx}\}_{k \in \nn}$ converges for all $\bx \in X'$ and in particular for $\bx = \tilde{\bx}$. Thus,
		\begin{equation*}
    		\lim_{k \rightarrow \infty} \norm{\bx^{k} - \tilde{\bx}} = \lim_{k \rightarrow \infty} \norm{\bx^{j_{k}} - \tilde{\bx}} = 0,
		\end{equation*}
		which concludes the proof.
	\end{proof}
	
\section{Numerical Results} \label{Sec:numeric}
	This section is devoted to evaluate FBi-PG numerically and compare its performance against existing algorithms that are proposed to solve convex bi-level optimization problems. For this purpose, we will follow the numerical experimental set-up of \cite{MS2023}. Here is the description of the numerical experiments. We conduct two experiments on real-world problems
	\begin{itemize}
		\item Songs release year prediction using linear regression
		\item Fake news classification using logistic regression.
	\end{itemize}
	In each experiment, we run FBi-PG with three parameters $\gamma = 1.3$, $\gamma = 1.5$ and $\gamma = 3$ (all with $a = 2$), which corresponds to three possible inner rates of convergence (see Theorem \ref{T:FastRateIn} and Proposition \ref{P:RateIn}), and compare it to the following algorithms
	\begin{itemize}
    	\item[$\rm{(i)}$] Bi-SG of \cite{MS2023} with $\alpha = 0.95$ and $c = 1$.
    	\item[$\rm{(ii)}$] FISTA \cite{BT2009} applied on the regularized problem \eqref{Prob:Tik} with a fixed regularization parameter $\alpha_{K} = 1/K$ when the number of iterations is $K = 10^{5}$.
	\end{itemize}
	It should be noted that due to the numerical results of \cite{MS2023}, which provide a numerical comparison between several existing algorithm for convex bi-level optimization problems, we present here only the Bi-SG \cite{MS2023}, which was shown to be superior in the experiments conducted there.
\medskip

	All experiments were run on 11th Gen Intel(R) Core(TM) i7-1195G7 @ 2.90GHz 2.92 GHz with a total RAM memory of 16GB and 4 physical cores (using a free account of Google colab notebook). We run each algorithm from the same starting point $\bx^{0} = \bo$ and provide details about both the inner and outer optimization problems. In terms of the inner optimization problem, for each method $1 \leq l \leq 5$, we record the following measure over iterations
	\begin{equation*}
		\Delta^{l}\varphi\left(k\right) := \varphi\left(\bx_{(l)}^{k}\right) - \varphi^{\ast},
	\end{equation*}	
	where $\bx_{(l)}^{k}$ is the point computed by the method $l$ after $k$ iterations, and $\varphi^{\ast}$ is the optimal value of $\varphi$, which we compute using an off-the-shelf solver\footnote{$https://scikit-learn.org/stable/modules/generated/sklearn.linear\_model.LogisticRegression$}.
\medskip
	
	As for the outer optimization problem, we present a plot that record the progress of the values of $\omega$ versus $-\log(\Delta^{l}\varphi\left(k\right))$, namely the number of accuracy digits, for an algorithm $l$ at iteration $k$. This plot aims at understanding the behaviour of the algorithm in solving both the inner and outer optimization problems. On the $x$-axis we plot the values of $-\log(\Delta^{l}\varphi\left(k\right))$, which is an indicator of how accurate the inner problem has been solved. On the $y$-axis we plot the values of the outer objective function $\omega$ as we get more accurate in solving the inner problem. Since the $x-$axis represents the negative of (the log of) the optimality gap of the inner-level problem, algorithms that were not able to achieve a good accuracy in solving the inner problem appear as a ``short" line in this plot. In between the algorithms that achieve a good accuracy in the inner optimality gap, we seek the algorithm that obtains the minimal value of $\omega$.

\subsection{Songs Release Year Prediction Using Linear Regression}	
	In this experiment, we used the YearPredictionMSD dataset\footnote{https://archive.ics.uci.edu/ml/datasets/yearpredictionmsd} to train a linear regression model to predict songs release years. We considered a sample of the data-set with $2000$ uniformly i.i.d randomly selected songs (without replacements). For each sample, we use $A$ to denote the feature matrix induced from it. Therefore, the $\left(i , j\right)$ entry of $A$ is the value of the $j$-th attribute of the $i$-th song in the sample. In addition, we use $b$ to denote the release year of the songs in the sample. That is, the $i$-th entry in $\bb$ is the release year of the $i$-th song in the sample. Our goal is to train a predictor by finding a weights vector $\bx$ that minimizes the following objective function
	\begin{equation*}
    	\varphi\left(\bx\right) = \frac{1}{2N}\norm{A\bx - \bb}^{2},
	\end{equation*}
	where $N$ is the number of songs in the sample. In this case, we consider the outer objective function $\omega \equiv \norm{\cdot}_{1}$. It should be noted that this inner objective function indeed satisfies the H\"olderian error bound property \cite{BNPS2015}.

	\begin{figure}
		\begin{subfigure}{.45\textwidth}
			\centering
			\includegraphics[width=1\linewidth]{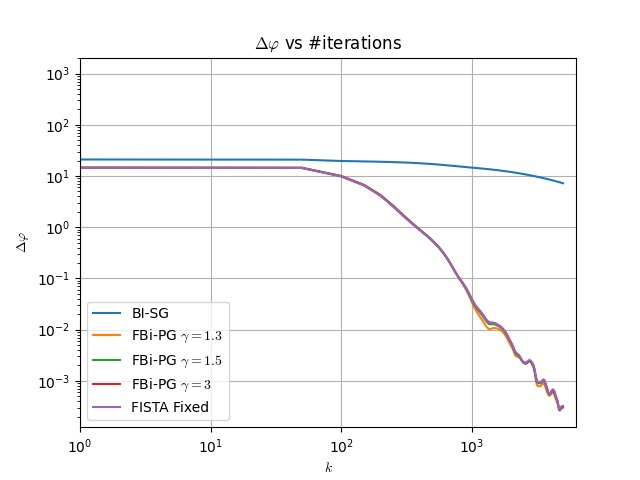}
			\caption{Value of the inner objective function vs. number of iterations}
		\end{subfigure}
		\hspace{0.1in}
		\begin{subfigure}{.45\textwidth}
			\centering
			\includegraphics[width=1\linewidth]{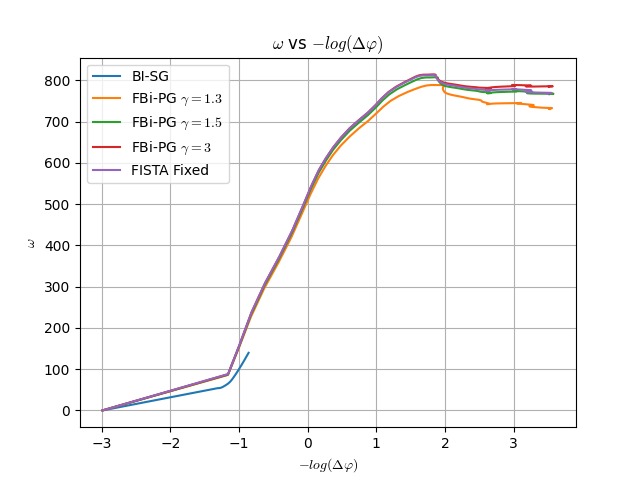}
			\caption{Number of accuracy digits of the inner problem vs. value of the outer objective function}
		\end{subfigure}
		\caption{The progress of the algorithms. In the right-hand side plot, the blue line is ``short" due to the low level of accuracy achieved in solving the inner problem as can be seen from the left-hand side plot.}
	\end{figure}
	
	In this experiment, we observe that the fixed regularization parameter we utilized performs comparably to the FBi-PG algorithm with all values of $\gamma$. However, as discussed in the paper, the major drawback of this approach is the inability of determining the ``right" regularization parameter in advance when solving a problem. This will be illustrated numerically in the forthcoming additional example.

\subsection{Fake News Classification Using Logistic Regression}
	Following the problem description in \cite{MS2023}, we immediately present the optimization model. We are given a matrix $A \in \real^{N \times m}$ and a vector $\bz \in \real^{N}$, which contains the labels of the articles. The corresponding logistic loss function is defined by				
	\begin{equation*}
		\varphi\left(\bx\right) \equiv \frac{1}{N}\sum_{i = 1}^{N}\left(\log\left(w\left(\bx^{T}\ba_{i}\right)\right)\bz_{i} + \log\left(1 - w\left(\bx^{T}\ba_{i}\right)\right)\left(1 - \bz_{i}\right) \right),
	\end{equation*}
	where $w : \real \rightarrow \real_{++}$ is the sigmoid function defined by $w\left(t\right) = \exp^{t}/\left(1 + \exp^{t}\right)$ and $\bz = \left(\bz_{1} , \bz_{2} , \ldots , \bz_{N}\right)^{T}$ for which $\bz_{i} \in \left\{0 , 1 \right\}$. We note that the logistic loss function satisfies the H\"olderian error bound property \cite{DS2023} and therefore our improved simultaneous rate and convergence guarantees proven in Theorems \ref{T:SimulHolder} and \ref{T:Convergence} are valid, respectively.
\medskip
		
	In this example, the bi-level approach can be used here to reduce the number of features in the vector $\bx$ obtained from minimizing the logistic loss $\varphi\left(\cdot\right)$. In other words, we wish to find the sparsest vector in between all the minimizers of the logistic regression function $\varphi\left(\cdot\right)$. Therefore, we take the following classical outer objective function
	\begin{equation*}	
    		\omega\left(\bx\right) = \norm{\bx}_{1}.
	\end{equation*}	
	In our experiment, we seek to differentiate between fake and genuine newspaper articles using a dataset from a Kaggle competition\footnote{Refer to $https://www.kaggle.com/competitions/fakenews/overeview$} comprising of 10,413 real and 10,387 fake news articles. Each article is accompanied by textual content and its authenticity status. We randomly uniformly select 500 fake and real newspapers, employing the pre-processing steps from \cite{MS2023}, and represent each article with 250 features as detailed in \cite{MS2023}.
	
	\begin{figure}
		\begin{subfigure}{.45\textwidth}
			\centering
			\includegraphics[width=1\linewidth]{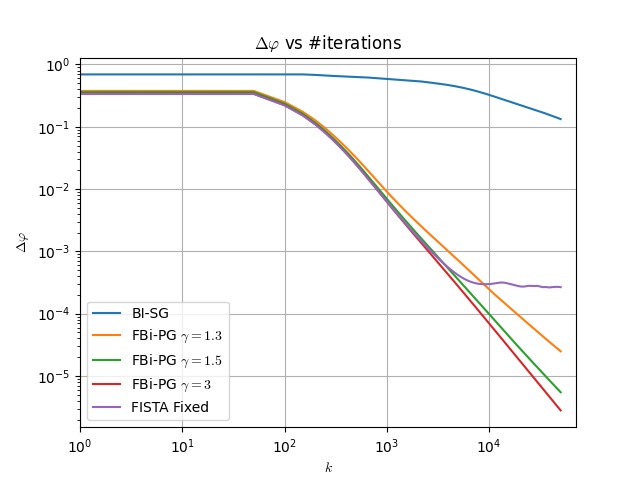}
			\caption{Value of the inner objective function vs. number of iterations}
		\end{subfigure}
		\hspace{0.1in}
		\begin{subfigure}{.45\textwidth}
			\centering
			\includegraphics[width=1\linewidth]{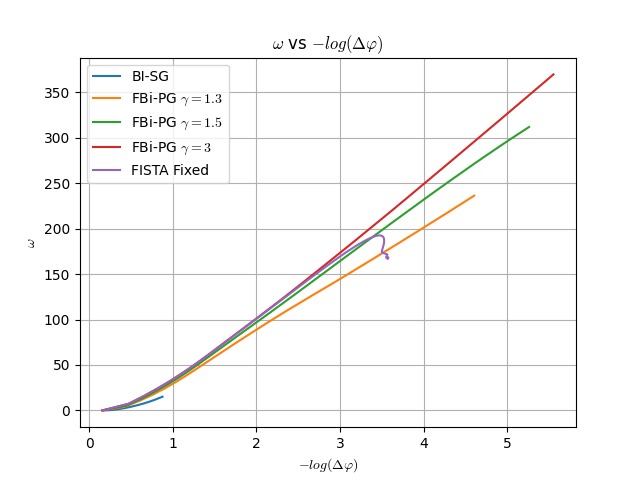}
			\caption{Number of accuracy digits of the inner problem vs. value of the outer objective function}
		\end{subfigure}
		\caption{The progress of the algorithms. In the left-hand side plot, the purple line, which is FISTA Fixed ``gets stuck" after $10^{4}$ iterations due to the fixed regularization parameter.}
	\end{figure}
	
	In this case, we see that the fixed regularization parameter that we choose doesn't solve the bi-level optimization problem since it ``gets stuck" when solving the inner problem. This means that for our choice of regularizing parameter, optimal solutions of the regularized problem don't coincide with optimal solutions of the bi-level optimization problem. However, we see in the left-hand side plot that all versions of FBi-PG solve the inner problem. Moreover, as our theory shows, the inner rate of FBi-PG is getting better as $\gamma$ increases.
\newpage

\section*{Appendix A: A Generic Rate for a Fixed Regularized Problem}
	We show a generic rate result  which is independent of the structure of the bi-level  problem and of any single-level optimization algorithm when using a fixed regularizing parameter. We describe this formally as follows.
	\begin{itemize}
		\item Let $F : \real^{n} \rightarrow \erl$ be a lower semicontinuous, proper and convex that is bounded from below and satisfies some other hypothesis, which we quantify by writing $F \in {\cal H}$.
 		\item Assume that a {\em fixed} regularizing parameter is known in advance for the regularized bi-level problem, namely,  the solution of the regularized problem coincides with the solution of the bi-level problem.
   		\item Any optimization algorithm ${\cal A}$ which generates iteratively a sequence $\Seq{\bx}{k}$ for minimizing $F\in {\cal H}$,  enjoys a proven convergence rate, in the sense that, for any $k \in \nn$, we have
			\begin{equation*}
				F\left(\bx^{k}\right) - F\left(\bx^{\ast}\right) \leq {\cal R}_{k},
			\end{equation*}
			where ${\cal R}_{k}$ is the rate of convergence and $\bx^{\ast}$ is a minimizer of $F$.
	\end{itemize}
	We will use this single-level algorithm to solve bi-level optimization problems. Let $\varphi$ and $\omega$ be the inner and outer objective functions of the bi-level problem, and for a predefined iteration index $K \in \nn$, consider the regularized objective  $F_{K} := \varphi + \alpha_{K}\omega$ where $\alpha_K>0$ is the fixed regularizing parameter known in advance, and such that $F_K \in {\cal H}$. Then, under this scenario, we have the following generic rate result.
	\begin{theorem}
		 Let $\bx^{K}$ be the output after running $K$ iterations of algorithm ${\cal A}$ on the function $F_{K}$. Then,
		 \begin{equation*}
		 	\varphi\left(\bx^{K}\right) - \varphi\left(\bx'\right) \leq {\cal R}_{K} + \alpha_{K}\left(\omega\left(\bx'\right) - \omega^{\ast}\right) \quad \text{and} \quad \omega\left(\bx^{K}\right) - \omega\left(\bx'\right) \leq \alpha_{K}^{-1}{\cal R}_{K},
		 \end{equation*}
		 where $\bx'$ is an optimal solution of the bi-level optimization problem and $\omega^{\ast} = \inf_{\bx \in \real^{n}} \omega\left(\bx\right)$.
	\end{theorem}
	\begin{proof}
Applying $K$ iterations of the algorithm ${\cal A}$ to minimize $F_{K}$ we obviously get that \begin{equation} \label{FKRate}
		F_{K}\left(\bx^{K}\right) - F_{K}\left(\bx_{K}^{\ast}\right) \leq {\cal R}_{K},
	\end{equation}
	where $\bx_{K}^{\ast}$ is a minimizer of $F_{K}$  which is also a minimizer of the bi-level problem since $\alpha_{K}$ is a
fixed regularization parameter choosen that way. Using the definition of $F_{K}$,  and since $\varphi\left(\bx^{K}\right) - \varphi\left(\bx'\right) \geq 0$, we immediately deduce from \eqref{FKRate} the following (recall that $\bx_{K}^{\ast} = \bx'$)
		\begin{equation*}
			\alpha_{K}\left(\omega\left(\bx^{K}\right) - \omega\left(\bx'\right)\right) = \varphi\left(\bx^{K}\right) - \varphi\left(\bx'\right) + \alpha_{K}\left(\omega\left(\bx^{K}\right) - \omega\left(\bx'\right)\right) \leq F_{K}\left(\bx^{K}\right) - F_{K}\left(\bx_{K}^{\ast}\right) \leq {\cal R}_{K},
		\end{equation*}
		which  proves the rate of convergence for the outer problem. For the inner problem, using again \eqref{FKRate} we get,		\begin{equation*}
			\varphi\left(\bx^{K}\right) - \varphi\left(\bx'\right) = F_{K}\left(\bx^{K}\right) - F_{K}\left(\bx'\right) + \alpha_{K}\left(\omega\left(\bx'\right) - \omega\left(\bx^{K}\right)\right) \leq {\cal R}_{K} + \alpha_{K}\left(\omega\left(\bx'\right) - \omega^{\ast}\right),
		\end{equation*}
		which completed the proof.
	\end{proof}

\section*{Appendix B: Proof of Lemma \ref{L:SumDelta}} \label{A:Tech}
	We begin with two technical lemmas that will be useful in proving the required result. Before doing so, we recall the definitions of the sequences $\seq{d}{k}$ and $\seq{\eta}{k}$, which will be used below, and are respectively defined by
	\begin{equation*}
		d_{k} := t_{k - 1}^{2} - \left(t_{k}^{2} - t_{k}\right) \quad \text{and} \quad \eta_{k} := t_{k - 1}^{2}\alpha_{k - 1} - \left(t_{k}^{2} - t_{k}\right)\alpha_{k}.
	\end{equation*}
	We begin with a result that will be used in the sequel.
    \begin{lemma} \label{L:SumOmega}
    	Let $\Seq{\bx}{k}$ be a sequence generated by FBi-PG with $1 < \gamma < 2$ and $a > 2$. Suppose that $\seq{b}{k} \subseteq \real_{+}$ is a sequence satisfying $b_{k} \leq b\eta_{k}$ for all $k \in \nn$ and some $b > 0$. Then, there exists $M > 0$ such that $\sum_{s = 0}^{k - 1} b_{s}\left(\omega\left(\bx'\right) - \omega\left(\bx^{s}\right)\right) \leq M$ for all $k \in \nn$ and $\bx' \in X'$.
    \end{lemma}
    \begin{proof}
        Let $k \in \nn$. From Theorem \ref{T:SimulHolder}(ii), we have for all $k \in \nn$ that
        \begin{equation} \label{L:SumOmega:1}
            \sum_{s = 0}^{k - 1} b_{s}\left(\omega\left(\bx'\right) - \omega\left(\bx^{s}\right)\right) \leq \sum_{s = 0}^{k - 1} b_{s}\left[\omega\left(\bx'\right) - \omega\left(\bx^{s}\right)\right]_{+} \leq \frac{C}{a^{2}}\sum_{s = 0}^{k - 1} b_{s}\left(s + 1\right)^{-1},
        \end{equation}
        where $[x]_{+} \equiv \max\{0 , x\}$. Using \eqref{Etaproperty}, we have that $b_{s} \leq b\eta_{s} < b\left(s + 1\right)^{1 - \gamma}/2$ for all $s \in \nn$. Thus, we obtain that $b_{s}\left(s + 1\right)^{-1} \leq b\left(s + 1\right)^{-\gamma}/2$. Now, from Lemma \ref{L:TechSum} with $n_{1} = 2$, $n_{2} = k$ and $r = \gamma > 1$, we get that
        \begin{equation} \label{L:SumOmega:2}
            \sum_{s = 0}^{k - 1} \left(s + 1\right)^{-\gamma} = \sum_{s = 1}^{k} s^{-\gamma} = 1 + \sum_{s = 2}^{k} s^{-\gamma} \leq 1 + \frac{1^{1 - \gamma} - k^{1 - \gamma}}{\gamma - 1} \leq 1 + \frac{1}{\gamma - 1} = \frac{\gamma}{\gamma - 1}.
        \end{equation}
        Combining \eqref{L:SumOmega:1} with \eqref{L:SumOmega:2} yields the desired result.
    \end{proof}
    Before proving Lemma \ref{L:SumDelta} about the summability of the sequence $\left\{ t_{k - 1}\delta_{k - 1} \right\}_{k \in \nn}$, we prove the summability of the sequence $\left\{ d_{k}\left(\varphi\left(\bx^{k}\right) - \varphi\left(\bx'\right)\right) \right\}_{k \in \nn}$.
	\begin{lemma} \label{L:SumPhi}
    	Let $\Seq{\bx}{k}$ be a sequence generated by FBi-PG with $1 < \gamma < 2$ and $a > 2$. Then, there exists $M > 0$ such that $\sum_{s = 0}^{k - 1} d_{s}\left(\varphi\left(\bx^{s}\right) - \varphi\left(\bx'\right)\right) < M$ for all $k \in \nn$ and $\bx' \in X'$.
	\end{lemma}
	\begin{proof}
    	Let $k \in \nn$ and $\bx' \in X'$. From the proof of Proposition \ref{P:SimIne} (\cf \eqref{P:SimIne:5}), we have that
   		\begin{equation*}
    		t_{k - 1}^{2}\left(F_{k - 1}\left(\bx^{k}\right) - F_{k - 1}\left(\bx'\right)\right) \leq \frac{\beta}{2}\norm{\bx^{0} - \bx'}^{2} - \sum_{s = 0}^{k - 1} d_{s}\left(\varphi\left(\bx^{s}\right) - \varphi\left(\bx'\right)\right) - \sum_{s = 0}^{k - 1} \eta_{s}\left(\omega\left(\bx^{s}\right) - \omega\left(\bx'\right)\right).
    	\end{equation*}
     	Rearranging this inequality yields that
     	\begin{equation} \label{L:SumPhi:1}
			t_{k - 1}^{2}\left(F_{k - 1}\left(\bx^{k}\right) - F_{k - 1}\left(\bx'\right)\right) + \sum_{s = 0}^{k - 1} d_{s}\left(\varphi\left(\bx^{s}\right) - \varphi\left(\bx'\right)\right) \leq \frac{\beta}{2}\norm{\bx^{0} - \bx'}^{2} + \sum_{s = 0}^{k - 1}\eta_{s}\left(\omega\left(\bx'\right) - \omega\left(\bx^{s}\right)\right).
     	\end{equation}
    	Since $\bx' \in X'$, it follows that $\varphi\left(\bx^{k}\right) - \varphi\left(\bx'\right) \geq 0$. Thus, combining with the fact that $F_{k - 1} \equiv \varphi + \alpha_{k - 1}\omega$, it follows from Theorem \ref{T:SimulHolder}(ii) that
    	\begin{equation*}
        	t_{k - 1}^{2}\left(F_{k - 1}\left(\bx^{k}\right) - F_{k - 1}\left(\bx'\right)\right) \geq \alpha_{k - 1}t_{k - 1}^{2}\left(\omega\left(\bx^{k}\right) - \omega\left(\bx'\right)\right) \geq -\alpha_{k - 1}t_{k - 1}^{2}\frac{C}{a^{2}\left(k + 1\right)}.
        \end{equation*}
        Using now the definitions $\alpha_{k}$ and $t_{k}$ (see \eqref{Bi-AGUpdate:0} and \eqref{D:Tk}, respectively) yields that
    	\begin{equation*}
        	t_{k - 1}^{2}\left(F_{k - 1}\left(\bx^{k}\right) - F_{k - 1}\left(\bx'\right)\right) \geq -\frac{C\left(k + a - 1\right)^{2 - \gamma}}{a^{4}\left(k + 1\right)} \geq -\frac{C\left(k + a - 1\right)}{a^{4}\left(k + 1\right)} \geq -\frac{C\left(ak + a\right)}{a^{4}\left(k + 1\right)} = -\frac{C}{a^{2}},
    	\end{equation*}
    	where the second inequality from the fact that $\left(k + a - 1\right)^{2 - \gamma} \leq k + a - 1$ since $2 - \gamma < 1$ (recall that $\gamma > 1$) and  $a - 1 > 0$ (recall that $a > 2$), and the last inequality follows from the fact that $k + a - 1 < k + a < ak + a$ since $a > 1$. By plugging it into \eqref{L:SumPhi:1} and adding $C/a^{3}$ to both sides, we get that
    	\begin{equation} \label{L:SumPhi:2}
			\sum_{s = 0}^{k - 1} d_{s}\left(\varphi\left(\bx^{s}\right) - \varphi\left(\bx'\right)\right) \leq \frac{C}{a^{3}} + \frac{\beta}{2}\norm{\bx^{0} - \bx'}^{2} + \sum_{s = 0}^{k - 1} \eta_{s}\left(\omega\left(\bx'\right) - \omega\left(\bx^{s}\right)\right).
    	\end{equation}
    	Applying Lemma \ref{L:SumOmega} with $b_{k} \equiv \eta_{k}$ and $b = 1$ it follows that there exists ${\tilde M} > 0$ such that $\sum_{s = 0}^{k - 1} \eta_{s}\left(\omega\left(\bx'\right) - \omega\left(\bx^{s}\right)\right) \leq {\tilde M}$. Hence, using this bound in  \eqref{L:SumPhi:2} yields the desired result.
    \end{proof}
	Now, we are ready to prove Lemma \ref{L:SumDelta}.
\medskip

    \noindent {\em Proof of Lemma \ref{L:SumDelta}.}
   		Let $k \in \nn$. From Proposition \ref{P:BasicIne} with $\bx = \bx^{k}$, we obtain, 
   \begin{equation} \label{L:SumDelta:1}
            \frac{\beta}{2}\norm{\bz^{k + 1} - \bx^{k}}^{2} - \frac{\beta}{2}\norm{\bz^{k} - \bx^{k}}^{2} \leq t_{k}^{2}\left(F_{k}\left(\bx^{k}\right) - F_{k}\left(\bx^{k + 1}\right)\right).
        \end{equation}
        From \eqref{Bi-AGUpdate:1} and the definition of the sequence $\bz^{k}$ we have that $\bz^{k + 1} - \bx^{k} = t_{k}\left(\bx^{k + 1} - \bx^{k}\right)$ and $\bz^{k} - \bx^{k} = \left(t_{k - 1} - 1\right)\left(\bx^{k} - \bx^{k - 1}\right)$. Therefore, using the definition of $\delta_{k}$ (see \eqref{DefDeltaK}), we obtain that
        \begin{align*}
			\frac{\beta}{2} \norm{\bz^{k + 1} - \bx^{k}}^{2} - \frac{\beta}{2} \norm{\bz^{k} - \bx^{k}}^{2} & = t_{k}^{2} \frac{\beta}{2}\norm{\bx^{k+1} - \bx^{k}}^{2} - (t_{k-1} - 1)^{2} \frac{\beta}{2}\norm{\bx^{k} - \bx^{k-1}}^{2} \nonumber \\
			& = \beta\left(t_{k}^{2}\delta_{k} - \left(t_{k - 1} - 1\right)^{2}\delta_{k - 1}\right) \nonumber \\
			& = \beta\left(t_{k}^{2}\delta_{k} - t_{k - 1}^{2}\delta_{k - 1}\right) + \beta\left(2t_{k - 1} - 1\right)\delta_{k - 1} \nonumber \\
			& \geq \beta\left(t_{k}^{2}\delta_{k} - t_{k - 1}^{2}\delta_{k - 1}\right) + \beta t_{k - 1}\delta_{k - 1},
        \end{align*}
        where the last inequality follows from the facts that $t_{k} \geq 1$ and $\delta_{k} \geq 0$ for all $k \in \nn$. Plugging this into \eqref{L:SumDelta:1} with $k = s$ yields that
        \begin{equation} \label{L:SumDelta:2}
            \beta\left(t_{s}^{2}\delta_{s} - t_{s - 1}^{2}\delta_{s - 1}\right) + \beta t_{s - 1}\delta_{s - 1} \leq t_{s}^{2}\left(F_{s}\left(\bx^{s}\right) - F_{k}\left(\bx^{s + 1}\right)\right).
        \end{equation}
        Summing \eqref{L:SumDelta:2} for $s = 0 , 1 , \ldots , k - 1$, we get that
        \begin{equation*}
            \beta\sum_{s = 0}^{k - 1} t_{s - 1}\delta_{s - 1} + \beta\left(t_{k - 1}^{2}\delta_{k - 1} - t_{-1}^{2}\delta_{-1}\right) = \beta\sum_{s = 0}^{k - 1} \left(t_{s}^{2}\delta_{s} - t_{s - 1}^{2}\delta_{s - 1} + t_{s - 1}\delta_{s - 1}\right) \leq \sum_{s = 0}^{k - 1} t_{s}^{2}\left(F_{s}\left(\bx^{s}\right) - F_{s}\left(\bx^{s + 1}\right)\right).
	   	\end{equation*}
	   	Since $t_{-1} = 0$ and $t_{k - 1}^{2}\delta_{k - 1} \geq 0$ we obtain that
        \begin{equation*}
            \beta\sum_{s = 0}^{k - 1} t_{s - 1}\delta_{s - 1} \leq \sum_{s = 0}^{k - 1} t_{s}^{2}\left(F_{s}\left(\bx^{s}\right) - F_{s}\left(\bx^{s + 1}\right)\right).
	   	\end{equation*}
	   	To complete the proof we will prove that the sum on the right-hand side is bounded by some $M > 0$ for any $k \in \nn$. By defining $\Phi_{k} = \varphi\left(\bx^{k}\right) - \varphi\left(\bx'\right)$ and $\Omega_{k} = \omega\left(\bx'\right) - \omega\left(\bx^{k}\right)$ for all $k \in \nn$, we obtain from the definition of $F_{k} \equiv \varphi + \alpha_{k}\omega$ that
        \begin{equation*}
            F_{k}\left(\bx^{k}\right) - F_{k}\left(\bx^{k + 1}\right) = F_{k}\left(\bx^{k}\right) - F_{k}\left(\bx'\right) + F_{k}\left(\bx'\right) - F_{k}\left(\bx^{k + 1}\right) = \Phi_{k} - \Phi_{k + 1} - \alpha_{k}\left(\Omega_{k} - \Omega_{k + 1}\right),
		\end{equation*}
		and hence
		\begin{align}
            \sum_{s = 0}^{k - 1} t_{s}^{2}\left(F_{s}\left(\bx^{s}\right) - F_{s}\left(\bx^{s + 1}\right)\right) & = \sum_{s = 0}^{k - 1} t_{s}^{2}\left(\Phi_{s} - \Phi_{s + 1} - \alpha_{s}\left(\Omega_{s} - \Omega_{s + 1}\right)\right) \nonumber \\
            & = t_{0}^{2}\Phi_{0} - t_{k - 1}^{2}\Phi_{k} + \sum_{s = 1}^{k - 1} \left(t_{s}^{2} - t_{s - 1}^{2}\right)\Phi_{s} - \alpha_{0}t_{0}^{2}\Omega_{0} + \alpha_{k - 1}t_{k - 1}^{2}\Omega_{k} \nonumber \\
            & - \sum_{s = 1}^{k - 1} \left(\alpha_{s}t_{s}^{2} - \alpha_{s - 1}t_{s - 1}^{2}\right)\Omega_{s} \nonumber \\
            & \leq t_{0}^{2}\Phi_{0} - \alpha_{0}t_{0}^{2}\Omega_{0} + \frac{a}{a - 2}\sum_{s = 1}^{k - 1} d_{s}\Phi_{s} + \alpha_{k - 1}t_{k - 1}^{2}\Omega_{k} + \sum_{s = 1}^{k - 1} \left(\alpha_{s - 1}t_{s - 1}^{2} - \alpha_{s}t_{s}^{2}\right)\Omega_{s}, \nonumber
        \end{align}
        where the last inequality follows from the fact that $-t_{k - 1}^{2}\Phi_{k} \leq 0$ (since $\Phi_{k} \geq 0$) and that
        \begin{equation} \label{L:SumDelta:21}
        	t_{s}^{2} - t_{s - 1}^{2} \leq t_{s} = \frac{s + a}{a} = \frac{a^{2}d_{s} - 1}{a\left(a - 2\right)} \leq \frac{ad_{s}}{a - 2},
        \end{equation}
        where the second equality follows from the definition of $d_{k}$ since
         \begin{equation*}
            d_{s} = t_{s - 1}^{2} - \left(t_{s}^{2} - t_{s}\right) = a^{-2}\left(\left(s + a - 1\right)^{2} - \left(s + a\right)s\right) = a^{-2}\left(\left(s + a\right)\left(a - 2\right) + 1\right).
        \end{equation*}
        Now, we will show that there exists $M > 0$ such that for all $k \in \nn$ we have that
        \begin{equation} \label{L:SumDelta:3}
        	\frac{a}{a - 2}\sum_{s = 1}^{k - 1} d_{s}\Phi_{s} + \alpha_{k - 1}t_{k - 1}^{2}\Omega_{k} + \sum_{s = 1}^{k - 1} \left(\alpha_{s - 1}t_{s - 1}^{2} - \alpha_{s}t_{s}^{2}\right)\Omega_{s} \leq M.
        \end{equation}
        We will prove that each of three terms above is bounded from above. First, from Lemma \ref{L:SumPhi} it follows that there exists $M_{1} > 0$ such that $\sum_{s = 0}^{k - 1} d_{s}\Phi_{s} \leq M_{1}$. Second, using Theorem \ref{T:SimulHolder}(ii) and the definition of $\Omega_{k}$ we obtain that
        \begin{equation} \label{L:SumDelta:BoundXi}
            \alpha_{k - 1}t_{k - 1}^{2}\Omega_{k} = \alpha_{k - 1}t_{k - 1}^{2}\left(\omega\left(\bx'\right) - \omega\left(\bx^{k}\right)\right) \leq \frac{\alpha_{k - 1}t_{k - 1}^{2}C}{a^{2}\left(k + 1\right)} = \frac{\left(k + a - 1\right)^{2 - \gamma}C}{a^{4}\left(k + 1\right)} \leq \frac{\left(k + a\right)C}{a^{4}\left(k + 1\right)} \leq \frac{C}{a^{3}},
        \end{equation}
        where the second equality from the definitions of $\alpha_{k}$ and $t_{k}$ (see \eqref{Bi-AGUpdate:0} and \eqref{D:Tk}, respectively), the second inequality follows from the fact that $2 - \gamma < 1$ since $\gamma > 1$, and the last inequality from the fact that $k + a < ak + a = a\left(k + 1\right)$ since $a > 2$. For the third term in \eqref{L:SumDelta:3}, we first split it as follows
        \begin{equation} \label{L:SumDelta:4}
        	\sum_{s = 1}^{k - 1} \left(\alpha_{s - 1}t_{s - 1}^{2} - \alpha_{s}t_{s}^{2} + \frac{a}{a - 2}\eta_{s}\right)\Omega_{s} - \frac{a}{a - 2}\sum_{s = 1}^{k - 1} \eta_{s}\Omega_{s}.
        \end{equation}
        Moreover, from \eqref{L:SumDelta:21} and the definitions of $d_{s}$ and $\eta_{s}$, we obtain that
        \begin{equation*}
        	\alpha_{s}t_{s} \leq \frac{a}{a - 2} \cdot \alpha_{s}d_{s} = \frac{a}{a - 2}\left(\alpha_{s}t_{s - 1}^{2} - \alpha_{s}\left(t_{s}^{2} - t_{s}\right)\right) \leq \frac{a}{a - 2}\left(\alpha_{s - 1}t_{s - 1}^{2} - \alpha_{s}\left(t_{s}^{2} - t_{s}\right)\right) = \frac{a}{a - 2}\eta_{s},
        \end{equation*}
        where the last inequality follows from the fact that $\alpha_{s} < \alpha_{s - 1}$. Using it and the definition of $\eta_{s}$ we obtain that
        \begin{equation*}
            \alpha_{s - 1}t_{s - 1}^{2} - \alpha_{s}t_{s}^{2} + \frac{a}{a - 2}\eta_{s} \geq \alpha_{s - 1}t_{s - 1}^{2} - \alpha_{s}t_{s}^{2} + \alpha_{s}t_{s} = \eta_{s} > 0,
        \end{equation*}
        and
        \begin{equation*}
            \alpha_{s - 1}t_{s - 1}^{2} - \alpha_{s}t_{s}^{2} + \frac{a}{a - 2}\eta_{s} = \eta_{s} - \alpha_{s}t_{s} + \frac{a}{a - 2}\eta_{s} \leq  \eta_{s} + \frac{a}{a - 2}\eta_{s} = \frac{2a - 2}{a - 2}\eta_{s},
        \end{equation*}
        where the inequality follows from the fact that $\alpha_{s}t_{s} \geq 0$. Therefore, by applying Lemma \ref{L:SumOmega} with $b_{k} \equiv \alpha_{k - 1}t_{k - 1}^{2} - \alpha_{k}t_{k}^{2} + a\eta_{k}/\left(a - 2\right) < c\eta_{k}$ with $b = \left(2a - 2\right)/\left(a - 2\right)$ we obtain that there exists $M_{3} > 0$ such that (recall that $t_{-1} = \eta_{0} = 0$ and $t_{0} = 1$)
        \begin{equation*}
        	\sum_{s = 1}^{k - 1} \left(\alpha_{s - 1}t_{s - 1}^{2} - \alpha_{s}t_{s}^{2} + \frac{a}{a - 2}\eta_{s}\right)\Omega_{s} = \sum_{s = 0}^{k - 1} \left(\alpha_{s - 1}t_{s - 1}^{2} - \alpha_{s}t_{s}^{2} + \frac{a}{a - 2}\eta_{s}\right)\Omega_{s} + \alpha_{0} \leq M_{3} + \alpha_{0}.
        \end{equation*}
        To complete the proof, we need to also bound the second term in \eqref{L:SumDelta:4}. Indeed, from Proposition \ref{P:SimIne} after rearranging and using the definitions of $\Omega_{k}$ and $F_{k - 1}$ we have (recall that $\eta_{0} = 0$)
        \begin{align*}
        	-\sum_{s = 1}^{k - 1} \eta_{s}\Omega_{s} & = -\sum_{s = 0}^{k - 1} \eta_{s}\Omega_{s} \leq t_{k - 1}^{2}\left(F_{k - 1}\left(\bx'\right) - F_{k - 1}\left(\bx^{k}\right)\right) + \frac{\beta}{2}\norm{\bx^{0} - \bx'}^{2} \\
        	& = t_{k - 1}^{2}\left(\varphi\left(\bx'\right) - \varphi\left(\bx^{k}\right) + \alpha_{k -1}\left(\omega\left(\bx'\right) - \omega\left(\bx^{k}\right)\right)\right) + \frac{\beta}{2}\norm{\bx^{0} - \bx'}^{2} \\
        	& \leq t_{k - 1}^{2}\alpha_{k -1}\Omega_{k} + \frac{\beta}{2}\norm{\bx^{0} - \bx'}^{2} \leq \frac{C}{a^{3}} + \frac{\beta}{2}\norm{\bx^{0} - \bx'}^{2},
        \end{align*}
        where the second inequality follows from the fact that $\varphi\left(\bx'\right) \leq \varphi\left(\bx^{k}\right)$ and last inequality follows from \eqref{L:SumDelta:BoundXi}. This completes the proof of the boundedness of \eqref{L:SumDelta:3}.

\bibliographystyle{plain}
\bibliography{notes}

\def\cprime{$'$}
\begin{thebibliography}{10}

\bibitem{AT2006}
A.~Auslender and M.~Teboulle.
\newblock Interior gradient and proximal methods for convex and conic
  optimization.
\newblock {\em SIAM Journal on Optimization}, 16(3):697--725, 2006.

\bibitem{BC2011-B}
H.~H. Bauschke and P.~L. Combettes.
\newblock {\em Convex Analysis and Monotone Operator Theory in {H}ilbert
  Spaces}.
\newblock CMS Books in Mathematics/Ouvrages de Math\'ematiques de la SMC.
  Springer, New York, 2011.

\bibitem{BT2009}
A.~Beck and M.~Teboulle.
\newblock A fast iterative shrinkage-thresholding algorithm for linear inverse
  problems.
\newblock {\em SIAM Journal on Imaging Sciences}, 2(1):183--202, 2009.

\bibitem{BT10}
A.~Beck and M.~Teboulle.
\newblock Gradient-based algorithms with applications to signal-recovery
  problems.
\newblock In {\em Convex optimization in signal processing and communications},
  pages 42--88. Cambridge Univ. Press, Cambridge, 2010.

\bibitem{BNPS2015}
J.~Bolte, T.~P. Nguyen, J.~Peypouquet, and B.~Suter.
\newblock From error bounds to the complexity of first-order descent methods
  for convex functions.
\newblock {\em Mathematical Programming}, 165(2):471--507, 2017.

\bibitem{CJYM2024}
J.~Cao, R.~Jiang, E.~H. Yazdandoost, and A.~Mokhtari.
\newblock An accelerated gradient method for convex smooth simple bilevel
  optimization.
\newblock In {\em Advances in Neural Information Processing Systems},
  volume~37, pages 45126--45154, 2024.

\bibitem{CD2015}
A.~Chambolle and C.~Dossal.
\newblock On the convergence of the iterates of the ``fast iterative
  shrinkage/thresholding algorithm''.
\newblock {\em Journal of Optimization Theory and Applications}, 166:968--982,
  2015.

\bibitem{CSJW2024}
P.~Chen, X.~Shi, R.~Jiang, and J.~Wang.
\newblock Penalty-based methods for simple bilevel optimization under
  h\"{o}lderian error bounds.
\newblock In {\em Advances in Neural Information Processing Systems},
  volume~37, pages 140731--140765, 2024.

\bibitem{DZ2020}
S.~Dempe and A.~Zemkoho.
\newblock {\em Bilevel Optimization}, volume 161 of {\em Springer Optimization
  and Its Applications}.
\newblock Springer, Berlin, 2020.

\bibitem{DS2023}
L.~Doron and S.~Shtern.
\newblock Methodology and first-order algorithms for solving nonsmooth and
  non-strongly convex bilevel optimization problems.
\newblock {\em Mathematical Programming}, 201:521--558, 2023.

\bibitem{FT2008}
M.~P. Friedlander and P.~Tseng.
\newblock Exact regularization of convex programs.
\newblock {\em SIAM Journal on Optimization}, 18(4):1326--1350, 2008.

\bibitem{JAMH2023}
R.~Jiang, N.~Abolfazli, A.~Mokhtari, and E.~Y. Hamedani.
\newblock A conditional gradient-based method for simple bilevel optimization
  with convex lower-level problem.
\newblock In {\em International Conference on Artificial Intelligence and
  Statistics}, pages 10305--10323. PMLR, 2023.

\bibitem{LTVP2025}
P.~Latafat, A.~Themelis, S.~Villa, and P.~Patrinos.
\newblock On the convergence of proximal gradient methods for convex simple
  bilevel optimization.
\newblock {\em Journal of Optimization Theory and Applications}, 204(3):51,
  2025.

\bibitem{MS2023}
R.~Merchav and S.~Sabach.
\newblock Convex bi-level optimization problems with nonsmooth outer objective
  function.
\newblock {\em SIAM Journal on Optimization}, 33(4):3114--3142, 2023.

\bibitem{N1983}
Y.~Nesterov.
\newblock A method for solving the convex programming problem with convergence
  rate \( o(1/k^2) \).
\newblock {\em Soviet Mathematics Doklady}, 27(2):372--376, 1983.

\bibitem{SS2017}
S.~Sabach and S.~Shtern.
\newblock A first order method for solving convex bilevel optimization
  problems.
\newblock {\em SIAM J. Optim.}, 27(2):640--660, 2017.

\bibitem{SBY2024}
S.~Samadi, D.~Burbano, and F.~Yousefian.
\newblock Achieving optimal complexity guarantees for a class of bilevel convex
  optimization problems.
\newblock In {\em 2024 American Control Conference (ACC)}, pages 2206--2211,
  2024.

\bibitem{S07}
M.~Solodov.
\newblock An explicit descent method for bilevel convex optimization.
\newblock {\em J. Convex Anal.}, 4(2):227--237, 2007.

\bibitem{T2018}
M.~Teboulle.
\newblock A simplified view of first order methods for optimization.
\newblock {\em Mathematical Programming}, 170:67--96, 2018.

\bibitem{T2021-T}
A.~Teush.
\newblock Proximal gradient methods for finding minimal norm-like solutions in
  convex optimization problems.
\newblock 2021.
\newblock MSc Thesis, School of Mathematical Sciences, Tel Aviv University.
  \url{https://tau.primo.exlibrisgroup.com/permalink/972TAU_INST/quev9q/alma9933339452504146}.

\bibitem{TA77-B}
A.~N. Tikhonov and V.~Y. Arsenin.
\newblock {\em Solutions of ill-posed problems}.
\newblock V. H. Winston \& Sons, Washington, D.C.: John Wiley \& Sons, New
  York-Toronto, Ont.-London, 1977.
\newblock Translated from the Russian, Preface by translation editor Fritz
  John, Scripta Series in Mathematics.

\bibitem{WSJ2024}
J.~Wang, X.~Shi, and R.~Jiang.
\newblock Near-optimal convex simple bilevel optimization with a bisection
  method.
\newblock In {\em Proceedings of The 27th International Conference on
  Artificial Intelligence and Statistics}, volume 238 of {\em Proceedings of
  Machine Learning Research}, pages 2008--2016. PMLR, 2024.

\bibitem{ZCXZ2024}
H.~Zhang, L.~Chen, J.~Xu, and J.~Zhang.
\newblock Functionally constrained algorithm solves convex simple bilevel
  problem.
\newblock In {\em Advances in Neural Information Processing Systems},
  volume~37, pages 57591--57618, 2024.

\end{thebibliography}
	
\end{document}